\documentclass[12pt]{amsart}
\usepackage[headings]{fullpage}
\usepackage{amssymb,epic,eepic,epsfig,amsbsy,amsmath,amscd}
\usepackage[all]{xy}
\usepackage{color}
\numberwithin{equation}{section}
                        \theoremstyle{plain}
\usepackage{mathrsfs}

\newcommand{\psdraw}[2]
         {\begin{array}{c} \hspace{-1.3mm}
         \raisebox{-4pt}{\psfig{figure=#1.eps,width=#2}}
         \hspace{-1.9mm}\end{array}}

\newtheorem{theorem}{Theorem}[section]
\newtheorem{thm}{Theorem}
\newtheorem{lemma}[theorem]{Lemma}
\newtheorem{corollary}[theorem]{Corollary}
\newtheorem{proposition}[theorem]{Proposition}
\newtheorem{conjecture}{Conjecture}

\theoremstyle{definition}
\newtheorem{remark}[theorem]{Remark}

\theoremstyle{plain}

\newtheorem{prop}{Proposition}

\theoremstyle{definition}

\newtheorem{defn}[prop]{Definition}

\theoremstyle{remark}

\newcommand{\eqM}{\overset{M}{=}}

\newcommand{\lcr}{\raisebox{-5pt}{\mbox{}\hspace{1pt}
                  \epsfig{file=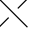}\hspace{1pt}\mbox{}}}
\newcommand{\ift}{\raisebox{-5pt}{\mbox{}\hspace{1pt}
                  \epsfig{file=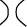}\hspace{1pt}\mbox{}}}
\newcommand{\zer}{\raisebox{-5pt}{\mbox{}\hspace{1pt}
                  \epsfig{file=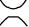}\hspace{1pt}\mbox{}}}

\def\BC{\mathbb C}

\def\BZ{\mathbb Z}

\def\BT{\mathbb T}
\def\BQ{\mathbb Q}

\def\CA{\mathcal A}

\def\CM{\mathcal M}
\def\CO{\mathcal O}
\def\CP{\mathcal P}
\def\CR{\mathcal R}
\def\CS{\mathcal S}
\def\CT{\mathcal T}

\def\fa{\mathfrak a}
\def\fb{\mathfrak b}

\def\fm{\mathfrak m}
\def\fp{\mathfrak p}
\def\ft{\mathfrak t}

\def\fb{\mathfrak b}
\def\fp{\mathfrak p}
\def\ft{\mathfrak t}

\def\la{\langle}
\def\ra{\rangle}
\def\ov{\bar}

\DeclareMathOperator{\tr}{\mathrm tr}

\def\al{\alpha}
\def\ve{\varepsilon}
\def\be { \begin{equation} }
\def\ee { \end{equation} }

\def\bD{{\bar D }}
\def\bTheta{\bar \Theta  }
\def\btheta{\tilde \theta  }

\def\im{\mathrm {Im}}

\def\bbtheta{\bar \theta}
\def\bbs{\bar {\mathfrak s}}
\def\bt{\tilde {\mathfrak t}}

\def\bS{\bar \CS}
\def\bp{{\tilde {\mathfrak p}}}
\def\bt{{\tilde {\mathfrak t}}}
\def\wt{\widetilde}
\def\bbp{\bar {\mathfrak p   }}
\def\bbt{\bar {\mathfrak t   }}

\def\CC{\mathcal C}
\newcommand\no[1]{}

\def\bP{\bar \CP}

\def\bT{\bar \CT}

\def\be{\begin{equation}}
\def\ee{\end{equation}}
\def\bes{\begin{equation*}}
\def\ees{\end{equation*}}
\def\ba{\begin{align}}
\def\ea{\end{align}}
\def\bas{\begin{align*}}
\def\eas{\end{align*}}

\def\fp{\mathfrak p}

\def\ve{\varepsilon}
\def\la{\langle}
\def\ra{\rangle}

\def\coeff{\operatorname{coeff}}

\def\tP{\tilde P}
\def\tQ{\tilde Q}

\def\tCS{\tilde{\mathcal S}}

\def\bS{\bar {\mathcal S}}
\def\CS{ {\mathcal S}}
\def\bD{\bar D}
\def\bN{\mathscr N}
\def\scP{\mathscr P}
\def\scQ{\mathscr Q}
\def\scX{\mathscr X}
\def\fm{\mathbf m}

\def\eqlt{\, \overset {\mathrm{lt}}{=} \ }

\begin{document}

\title{On the AJ conjecture for knots}

\author[Thang  T. Q. Le]{Thang  T. Q. Le}
\address{School of Mathematics, 686 Cherry Street,
 Georgia Tech, Atlanta, GA 30332, USA}
\email{letu@math.gatech.edu}

\author[Anh T. Tran]{Anh T. Tran}
\address{Department of Mathematics, The Ohio State University, Columbus, OH 43210, USA}
\email{tran.350@osu.edu}

\address{Faculty of Mathematics and Informatics, University of Natural Sciences, Vietnam National University, 227 Nguyen Van Cu, District 5, Ho Chi Minh City, Vietnam}
\email{hqvu@hcmuns.edu.vn}

\thanks{T.L. was supported in part by National Science Foundation. \\
2010 {\em Mathematics Classification:} Primary 57N10. Secondary 57M25.}

\dedicatory{\rm{With an appendix written jointly with VU Q. HUYNH}}

\begin{abstract}
We confirm the AJ conjecture \cite{Ga04} that relates the $A$-polynomial and the colored Jones polynomial for  hyperbolic knots satisfying certain conditions. In particular, we show that the conjecture holds true for some classes of two-bridge knots and pretzel knots. This extends the result of the first author in \cite{Le06} where he established the AJ conjecture for a large class of two-bridge knots, including all twist knots. Along the way, we explicitly calculate the universal $SL_2$-character ring of the knot group of the $(-2,3,2n+1)$-pretzel knot and show that it is reduced for all integers $n$.
\end{abstract}

\maketitle

\setcounter{section}{-1}

\section{Introduction}

\subsection{The AJ conjecture}

For a knot $K$ in $S^3$, let $J_K(n) \in \mathbb Z[t^{\pm 1}]$ be the colored Jones polynomial of $K$ colored by the (unique) $n$-dimensional simple representation of $sl_2$ \cite{Jo, RT}, normalized so that
for the unknot $U$,
$$J_U(n) = [n] := \frac{t^{2n}- t^{-2n}}{t^2 -t^{-2}}.$$
The color $n$ can be assumed to take negative integer values by setting $J_K(-n) = - J_K(n)$. In particular, $J_K(0)=0$. It is known that $J_K(1)=1$ and $J_K(2)$ is the ordinary Jones polynomial.

Define two linear operators $L,M$ acting on the set of discrete functions $f: \mathbb Z \to \mathbb \CR:=\BC[t^{\pm 1}]$ by
$$(Lf)(n) := f(n+1), \qquad (Mf )(n) := t^{2n} f(n).$$
It is easy to see that $LM= t^2 ML$. The inverse operators $L^{-1}, M^{-1}$ are well-defined. One can consider $L,M$ as elements of the quantum torus
$$ \mathcal T := \mathbb \CR\langle L^{\pm1}, M^{\pm 1} \rangle/ (LM - t^2 ML),$$
which is not commutative, but almost commutative.

Let $$\mathcal A_K := \{ P \in \mathcal T \mid  P J_K=0\}.$$ It is a left ideal of $\mathcal T$, called the {\em recurrence ideal} of $K$. It was proved in \cite{GL} that for every knot $K$, the
recurrence ideal $\mathcal A_K$ is non-zero. Partial results were obtained
earlier by Frohman, Gelca and Lofaro through their theory of
non-commutative $A$-ideals \cite{FGL,Ge}. An element in $\mathcal A_K$ is called a recurrence relation for the colored Jones polynomial of $K$.

The ring $\mathcal T$ is not a principal left ideal domain, i.e. not every left ideal of $\mathcal T$ is generated by one element. By adding all inverses of polynomials in $t,M$ to
$\mathcal T$ one gets a principal left ideal domain $\tilde\CT$, cf. \cite{Ga04}. Denote the
generator of the extension $\tilde\CA_K:=\mathcal A_K \cdot \tilde\CT$ by $\alpha_K$. The element $\alpha_K$ can be presented in the form
$$\alpha_K(t;M,L) = \sum_{j=0}^{d} \alpha_{K,j}(t,M) \, L^j,$$
where the degree in $L$ is assumed to be minimal and all the
coefficients $\alpha_{K,j}(t,M)\in \BZ[t^{\pm1},M]$ are assumed to
be co-prime. The polynomial $\alpha_K$ is defined up to a polynomial in $\mathbb Z[t^{\pm 1},M]$. We call $\alpha_K$ the {\em recurrence polynomial} of $K$.

\medskip

Garoufalidis \cite{Ga04} formulated the following conjecture (see also \cite{FGL, Ge}).

\begin{conjecture}{\bf (AJ conjecture)} For every knot $K$, $\alpha_K |_{t=-1}$ is equal to the $A$-polynomial, up to a factor depending on $M$ only.
\label{c1}
\end{conjecture}

In the definition of the $A$-polynomial \cite{CCGLS}, we also allow the abelian component of the character variety, see Section \ref{2}.

\subsection{Main results} Conjecture \ref{c1} was established for a large class of two-bridge knots, including
all twist knots, by the first author \cite{Le06} using skein theory. In this paper
we generalize his result as follows.

\begin{thm} Suppose $K$ is a knot satisfying all the following conditions:

(i) $K$ is hyperbolic,

(ii) the $SL_2$-character variety of $\pi_1(S^3\setminus K)$ consists of two irreducible components (one abelian and one non-abelian),

(iii) the universal $SL_2$-character ring of $\pi_1(S^3\setminus K)$ is reduced,

(iv) the localized skein module $\bar{\CS}$ of $S^3 \setminus K$ is finitely generated,
and

(v) the recurrence polynomial of $K$ has $L$-degree greater than 1.\\
Then the AJ conjecture holds true for $K$.

\label{t1}

\end{thm}

For the definition of the localized skein module $\bar{\CS}$ of $S^3 \setminus K$ in the  condition (iv) of Theorem \ref{t1}, see Section \ref{AJ}.

\begin{thm} The following knots satisfy all the conditions (i)--(v) of Theorem \ref{t1} and hence the AJ conjecture holds true for them.

(a) All pretzel knots of type $(-2,3,6n\pm 1)$, $n \in \mathbb Z$.

(b) All two-bridge knots for which the $SL_2$-character variety has exactly two irreducible components; these include

\begin{itemize}
\item all double twist knots of the form $J(k,l)$ (see Figure \ref{fig.10}) with $k \not= l$,
\item all two-bridge knots $\fb(p,m)$ with $m=3$, and
\item all two-bridge knots $\fb(p,m)$ with $p$ prime and $\gcd(\frac{p-1}{2},\frac{m-1}{2})=1$.
\end{itemize}
\label{t2}
\end{thm}
Here we use the notation $\fb(p,m)$ for two bridge knots from \cite{BZ}.

\begin{figure}
\setlength{\unitlength}{0.09mm}
\thicklines{
\begin{picture}(300,460)(80,-20)
\put(0,0){\line(0,1){440}}
\put(0,0){\line(1,0){125}}
\put(100,100){\line(1,0){25}}
\put(100,100){\line(0,1){125}}
\put(125,-25){\line(1,0){150}}
\put(125,-25){\line(0,1){150}}
\put(125,125){\line(1,0){150}}
\put(275,-25){\line(0,1){150}}
\put(100,225){\line(1,0){50}}
\put(150,225){\line(0,1){25}}
\put(125,250){\line(1,0){150}}
\put(125,250){\line(0,1){150}}
\put(125,400){\line(1,0){150}}
\put(275,400){\line(0,-1){150}}
\put(0,440){\line(1,0){150}}
\put(150,440){\line(0,-1){40}}
\put(190,30){{\large$l$}}
\put(185,310){\large{$k$}}
\put(250,400){\line(0,1){40}}
\put(250,440){\line(1,0){150}}
\put(400,440){\line(0,-1){440}}
\put(275,0){\line(1,0){125}}
\put(250,225){\line(1,0){50}}
\put(250,225){\line(0,1){25}}
\put(300,100){\line(0,1){125}}
\put(300,100){\line(-1,0){25}}
\end{picture}
}
\caption{The double twist knot $J(k,l)$. Here $k$ and $l$ denote
the numbers of half twists in the boxes. Positive (resp. negative) numbers correspond
to right-handed (resp. left-handed) twists. }
\label{fig.10}
\end{figure}
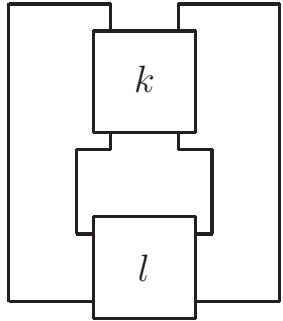

\begin{remark} Besides the infinitely many cases of two-bridge knots listed in Theorem \ref{t2}, explicit calculations seem to confirm that ``most two-bridge knots" satisfy the
conditions of Theorem \ref{t1} and hence AJ conjecture holds for them. In fact, among 155 $\fb(p,m)$ with  $p < 45$, only 9 hyperbolic knots $\fb(15,11)$, $\fb(21,13)$, $\fb(27,5)$, $\fb(27,17)$, $\fb(27,19)$, $\fb(33,5)$, $\fb(33,13)$, $\fb(33,23)$ and $\fb(35,29)$ do not satisfy the condition (ii) of Theorem \ref{t1}. Thus, the AJ conjecture holds for all two-bridge knots $\fb(p,m)$ with  $p < 45$ except for these 9 knots.
\end{remark}

\subsection{Other results} In our proof of Theorem \ref{t2}, it is important to know whether the universal $SL_2$-character ring of a knot group is reduced, i.e. whether its nil-radical is zero.
Although it is difficult to construct a group whose universal $SL_2$-character ring is not reduced (see \cite{LM}),
so far there are a few groups for which the universal $SL_2$-character ring is known to be reduced: free groups \cite{Si}, surface groups \cite{CM, Si}, two-bridge knot groups \cite{Le06, PS},
torus knot groups \cite{Mar}, two-bridge link groups \cite{LT}. In the present paper, we show that the universal $SL_2$-character ring of the $(-2,3,2n+1)$-pretzel knot is reduced for all integers $n$.

\subsection{Plan of the paper} We review skein modules and their relation with the colored Jones polynomial in Section 1. In Section 2 we prove some properties of the $SL_2$-character variety and the $A$-polynomial a knot. We discuss the role of localized skein modules in our approach to the AJ conjecture and give proofs of Theorems \ref{t1} and \ref{t2} in Section 3. In Section 4, we prove the reducedness of the universal $SL_2$-character ring of the $(-2,3,2n+1)$-pretzel knot. In Section 5, we prove that the localized skein module $\bar{\CS}$ of the $(-2,3,2n+1)$-pretzel knot is  finitely generated.  Finally we study the irreducibility of non-abelian $SL_2$-character varieties of two-bridge knots in the appendix.

\subsection{Acknowledgements} The authors would like to thank J. Etnyre, S. Garoufalidis, J. Marche, T. Mattman, K. Petersen, A. Sikora for helpful correspondence and discussions, and the referee for comments/suggestions.

\section{Skein Modules and the colored Jones polynomial}

\label{sec01}

In this section we will review skein modules and their relation with the colored Jones polynomial. The theory of the  Kauffman bracket skein module (KBSM) was introduced by Przytycki \cite{Pr} and Turaev \cite{Tu} as a generalization of the Kauffman bracket \cite{Ka} in $S^3$  to an arbitrary
$3$-manifold. The KBSM of a knot complement contains a lot, if not all, of information about its colored Jones polynomial.

\subsection{Skein modules}

 \label{subSkein}

 Recall that $\CR=\BC[t^{\pm1}]$.
 A {\em framed link} in an oriented $3$-manifold $Y$ is a disjoint union of embedded circles, equipped with a
 non-zero normal vector field. Framed links are considered up to isotopy. Let $\mathcal{L}$ be the set of isotopy
classes of framed links in the manifold $Y$, including the empty
link. Consider the free $\CR$-module with basis $\mathcal{L}$, and
factor it by the smallest submodule containing all expressions of
the form $\lcr-t\zer-t^{-1}\ift$ and
$\bigcirc+(t^2+t^{-2}) \emptyset$, where the links in each
expression are identical except in a ball in which they look like
depicted. This quotient is denoted by $\CS(Y)$ and is called the
Kauffman bracket skein module, or just skein module, of $Y$.

For an oriented surface $\Sigma$ we define $\CS(\Sigma):= \CS(Y)$, where $ Y= \Sigma \times [0,1]$ is the cylinder over $\Sigma$. The skein module
$\CS(\Sigma)$ has an algebra structure induced by the operation
of gluing one cylinder on top of the other. The operation of
gluing the cylinder over $\partial Y$ to $Y$ induces a
$\CS(\partial Y)$-left module structure on $\CS(Y)$.

\subsection{The skein module of $S^3$ and the colored Jones polynomial} When $Y=S^3$, the skein
module $\CS(Y)$ is free over $\CR$ of rank one, and is spanned by
the empty link. Thus if $\ell$ is a framed link in $S^3$, then
its value in the skein module $\CS(S^3)$ is $\langle \ell
\rangle$ times the empty link, where $\langle \ell \rangle \in
\CR$ is the Kauffman bracket of $\ell$ \cite{Ka} which is the Jones polynomial of the
{\em framed link} $\ell$ in a suitable normalization.

Let $S_n(z)$'s be the Chebychev polynomials defined by $S_0(z)=1$, $S_1(z)=z$ and $S_{n+1}(z)=zS_n(z)-S_{n-1}(z)$ for all $n \in \BZ$. For a framed knot $K$ in $S^3$ and an integer $n \ge 0$, we define the $n$-th
power $K ^n$ as the link consisting of $n$ parallel copies of $K$ (this is a 0-framing cabling operation).
Using these powers of a knot, $S_n(K)$ is defined as an element
of $\CS(S^3)$. We define the colored Jones polynomial $J_K(n)$ by the equation
$$ J_K(n+1):=(-1)^{n} \times\langle S_n(K) \rangle. $$
The $(-1)^n$ sign is added so that for the unknot $U$, $J_U(n) = [n].$ Then
$J_K(1)=1$ and $J_K(2)= - \langle K \rangle$. We extend the
definition for all integers $n$ by $J_K(-n)= -J_K(n)$ and
$J_K(0)=0$. In the framework of quantum invariants, $J_K(n)$ is
the $sl_2$-quantum invariant of $K$ colored by the $n$-dimensional
simple representation of $sl_2$.

\subsection{The skein module of the torus} Let $\BT^2$ be the torus with a fixed pair $(\mu, \lambda)$ of simple closed curves intersecting at exactly one point.
For co-prime integers $k$ and $l$, let $\lambda_{k,l}$ be a simple closed curve  on the torus homologically equal to $k\mu+ l\lambda$. It is not difficult to show that
the skein algebra $\CS(\BT^2)$ of the torus is generated, as an $\CR$-algebra, by all $\lambda_{k,l}$'s. In fact,
Bullock and Przytycki \cite{BP} showed that  $\CS(\BT^2)$ is
generated over $\CR$ by 3 elements $\mu,\lambda$ and $\lambda_{1,1}$, subject to some explicit relations.

Recall that $ \CT= \CR\la M^{\pm 1}, L^{\pm1} \ra/(LM - t^2 ML)$ is the quantum torus. Let $\sigma: \CT \to \CT$ be the involution defined by $\sigma(M^{k} L^{l}) := M^{-k} L^{-l}$.
Frohman and Gelca \cite{FG} showed that there is an algebra isomorphism $\Upsilon: \CS(\BT^2)\to \CT^{\sigma}$ given by
$$ \Upsilon(\lambda_{k,l}) :=  (-1)^{k+l} t^{kl} (M^{k}L^{l} +M^{-k}L^{-l}).$$
The fact that $\CS(\BT^2)$ and $\CT^{\sigma}$ are isomorphic algebras was also proved by Sallenave  \cite{Sa}.
\label{sec_torus}

\subsection{The orthogonal and peripheral ideals} Let $N(K)$ be a tubular neighborhood of an oriented knot $K$ in $S^3$, and $X$ the closure of $S^3 \setminus N(K)$. Then $\partial (N(K))= \partial(X)=
\BT^2$. There is a standard choice of a meridian $\mu$  and a longitude $\lambda$ on $\BT^2$ such that the linking
number between the longitude and the knot is zero. We use this pair $(\mu,\lambda)$ and the map $\Upsilon$ in the previous subsection to identify $\CS(\BT^2)$ with $\CT^\sigma$.

The torus $\BT^2= \partial(N(K))$ cut $S^3$ into two parts: $N(K)$ and $X$. We can consider $\CS(X)$ as a left $\CS(\BT^2)$-module and $\CS(N(K))$ as a right $\CS(\BT^2)$-module.
There is a bilinear bracket
$$ \la \cdot, \cdot \ra : \CS(N(K)) \otimes_{\CS(\BT^2)} \, \CS(X)\,  \to \CS(S^3)\equiv \CR$$
given by $\la \ell', \ell'' \ra := \la \ell' \cup \ell'' \ra$, where $\ell'$ and $\ell''$ are links in respectively $N(K)$ and $X$. Note that  if $\ell \in \CS(\BT^2)$, then
$$ \la \ell' \cdot \ell, \ell'' \ra = \la \ell',\ell\cdot \ell'' \ra.$$

In general $\CS(X)$ does not have an algebra structure, but it has the identity element--the empty link.
The map
$$ \Theta: \CS(\BT^2) \to \CS(X), \quad \Theta(\ell) := \ell \cdot \emptyset$$
is $\CS(\BT^2)$-linear. Its kernel $\CP:=\ker \Theta$ is called the {\em quantum peripheral ideal}, first introduced in \cite{FGL}. In \cite{FGL, Ge}, it was proved that every element in $\CP$ gives rise to a recurrence relation for the colored Jones polynomial.

The {\em orthogonal ideal} $\CO$ in \cite{FGL} is defined by
$$ \CO := \{ \ell\in \CS(\partial X)\quad  \mid \quad  \langle \ell' ,\Theta(\ell)
\rangle=0 \quad \text{for every }\ell' \in \CS(N(K))\}.$$
It is clear that $\CO$  is a left ideal of $\CS(\partial X)  \equiv \CT^{\sigma}$ and $\CP \subset \CO$. In \cite{FGL}, $\CO$ was called the formal ideal. According to \cite{Le06}, if $\CP=\CO$ for all knots then the colored
Jones polynomial distinguish the unknot from other knots.

\label{OP}

\subsection{Relation between the recurrence and orthogonal ideals}

As mentioned above, the skein algebra of the torus $\CS (\BT^2)$
can be identified with $\CT^{\sigma}$ via the $\CR$-algebra
isomorphism $\Upsilon$ sending $\mu,\lambda$ and $\lambda_{1,1}$ to respectively
$-(M+M^{-1}), -(L+L^{-1})$ and $t(ML+M^{-1}L^{-1})$.

\begin{proposition} One has $$(-1)^{n}\la  S_{n-1}(\lambda) , \Theta(\ell) \ra = \Upsilon(\ell)J_K(n)$$ for all $\ell \in \CS (\BT^2).$
\label{CP}
\end{proposition}

\begin{proof}
We know from the properties of the Jones-Wenzl idempotent (see e.g. \cite{Oh}) that
	\begin{eqnarray*}
		\la S_{n-1}(\lambda) \cdot \mu, \, \emptyset \ra  &=& (t^{2n}+t^{-2n}) \la S_{n-1}(\lambda),\emptyset \ra\\
		\la S_{n-1}(\lambda) \cdot \lambda, \, \emptyset \ra &=&  \la S_{n}(\lambda)+ S_{n-2}(\lambda), \emptyset \ra\\
		\la S_{n-1}(\lambda) \cdot \lambda_{1,1}, \, \emptyset \ra &=&  - \la t^{2n+1}S_{n}(\lambda)+ t^{-2n+1}S_{n-2}(\lambda), \emptyset \ra.
	\end{eqnarray*}
	
By definition $J_K(n)=(-1)^{n-1} \la S_{n-1}(\lambda), \emptyset \ra$. Moreover $(MJ_K)(n)=t^{2n}J_K(n)$ and $(LJ_K)(n)=J_K(n+1)$. Hence the above equations can be rewritten as
\begin{eqnarray*}
		(-1)^{n}\la  S_{n-1}(\lambda) , \, \Theta(\mu) \ra &=& -(M + M^{-1}) J_K(n)=\Upsilon(\mu) J_K(n)
			, \\
 		(-1)^{n}\la  S_{n-1}(\lambda) , \, \Theta(\lambda)\ra &=& -(L + L^{-1}) J_K(n)=\Upsilon(\lambda) J_K(n)
 			, \\
 		(-1)^{n}\la S_{n-1}(\lambda) , \, \Theta(\lambda_{1,1}) \ra &=& t(ML + M^{-1}L^{-1}) J_K(n)
 			= \Upsilon(\lambda_{1,1}) J(n).
\end{eqnarray*}
	
	Since $\CS (\BT^2)$ is generated by $\mu, \lambda$ and $\lambda_{1,1}$, we conclude that $$(-1)^{n}\la  S_{n-1}(\lambda) , \Theta(\ell) \ra = \Upsilon(\ell)J_K(n)$$ for all $\ell \in \CS (\BT^2).$
\end{proof}

\begin{corollary}
One has $\CO = \CA_K \cap \CT^{\sigma}.$
\label{th01}
\end{corollary}

\begin{proof}
Since $\{S_n(\lambda)\}_n$ generates the skein module $\CS (N(K))$, Proposition \ref{CP} implies that
\begin{eqnarray*}
\CO &=& \{ \ell\in \CS(\partial X)~\mid ~\langle \ell' ,\Theta(\ell)
\rangle=0 \quad \text{for every }\ell' \in \CS(N(K))\}\\
&=& \{ \ell\in \CS(\partial X)~\mid ~\langle S_n(\lambda) ,\Theta(\ell)
\rangle=0 \quad \text{for all integers~} n\}\\
&=& \{\ell\in \CS(\partial X)~\mid ~\Upsilon(\ell)J_K(n)=0 \quad \text{for all integers~} n\}.
\end{eqnarray*}
Hence $\CO=\CA_K \cap \CT^{\sigma}.$
\end{proof}

\begin{remark} 	Corollary \ref{th01} was already obtained in \cite{Ga08} by another method. Our proof uses the properties of the Jones-Wenzl idempotent only.
\end{remark}

\section{Character varieties and the $A$-polynomial}

\label{2}

For non-zero $f,g\in \BC[M,L]$, we say that $f$ is {\em
$M$-essentially equal to } $g$, and write $f \,\overset{M}{=}\,
g,$ if the quotient $f/g$ does not depend on $L$. We say that $f$ is {\em
$M$-essentially divisible by} $g$ if $f$ is $M$-essentially equal
to a polynomial divisible by $g$.

\subsection{The character variety of a group}

The set of representations of a finitely presented group $G$
into $SL_2(\BC)$ is an algebraic set defined over $\BC$, on which
$SL_2(\BC)$ acts by conjugation. The set-theoretic
quotient of the representation space by that action does not
have good topological properties, because two representations with
the same character may belong to different orbits of that action. A better
quotient, the algebro-geometric quotient denoted by $\chi(G)$
(see \cite{CS,LM}), has the structure of an algebraic
set. There is a bijection between $\chi(G)$ and the set of all
characters of representations of $G$ into $SL_2(\BC)$, hence
$\chi(G)$ is usually called the {\em character variety} of $G$.
For a manifold $Y$ we use $\chi(Y)$ also to denote $\chi(\pi_1(Y))$.

Suppose  $G=\BZ^2$, the free abelian group with two generators.
Every pair  of generators $\mu,\lambda$ will define an isomorphism
between $\chi(G)$ and $(\BC^*)^2/\tau$, where $(\BC^*)^2$ is the
set of non-zero complex pairs $(L,M)$ and $\tau$ is the involution
$\tau(M,L):=(M^{-1},L^{-1})$, as follows. Every representation is
conjugate to an upper diagonal one, with $M$ and $L$ being the
upper left entry of $\mu$ and $\lambda$ respectively. The
isomorphism does not change if one replaces $(\mu,\lambda)$ with
$(\mu^{-1},\lambda^{-1})$.

\subsection{The universal character ring} For a finitely presented group $G$,  the character variety $\chi(G)$ is determined by the traces of some fixed elements $g_1, \cdots, g_k$ in $G$. More precisely, one can find $g_1, \cdots, g_k$ in $G$ such that for every element $g \in G$ there exists a polynomial $\mathbf{P}_g$ in $k$ variables such that for any representation $r: G \to SL_2(\BC)$ one has $\tr(r(g)) = \mathbf{P}_g(x_1, \cdots, x_k)$ where $x_j:=\tr(r(g_j))$. The universal character ring of $G$ is then defined to be the quotient of the ring $\BC[x_1, \cdots, x_k]$ by the ideal generated by all expressions of the form $\tr(r(v))-\tr(r(w))$, where $v$ and $w$ are any two words in $g_1, \cdots, g_k$ which are equal in $G$. The universal character ring of $G$ is actually independent of the choice of $g_1, \cdots, g_k$. The quotient of the universal character ring of $G$ by its nil-radical is equal to the ring of regular functions on the character variety of $G$.

The universal character ring defined here is the skein algebra of $G$ of \cite{PS}, where it was proved that it is $TH(G)$ of Brumfiel-Hilden's book \cite{BH}. They proved that it is the universal character ring, which is defined as the coefficient algebra of the universal representation.

\label{universaldef}

\subsection{The $A$-polynomial}

\label{Zariski}

Let $X$ be the closure of $S^3$
minus a tubular neighborhood $N(K)$ of a knot $K$. The boundary of
$X$ is a torus whose fundamental group  is free abelian of rank
two. An orientation of $K$ will define a unique pair of an
oriented meridian  and an oriented longitude such that the linking
number between the longitude and the knot is zero, as in Subsection \ref{OP}. The pair provides
an identification of $\chi(\partial X)$ and $(\BC^*)^2/\tau$
which actually does not depend on the orientation of $K$.

The inclusion $\partial X \hookrightarrow X$ induces the restriction
map
$$\rho : \chi(X) \longmapsto \chi(\partial X)\equiv (\BC^*)^2/\tau.$$
 Let $Z$ be the image of
$\rho$ and  $\hat Z \subset (\BC^*)^2$ the lift of $Z$ under the
projection $(\BC^*)^2 \to (\BC^*)^2/\tau$. The Zariski closure of
$\hat Z\subset (\BC^*)^2 \subset \BC^2$ in $\BC^2$ is an algebraic set
consisting of components of dimension 0 or 1. The union of all the
one-dimension components is defined by a single polynomial $A_K\in
\BZ[M,L]$, whose coefficients are co-prime. Note that $A_K$ is
defined up to $\pm 1$. We call $A_K$ the \textit{$A$-polynomial} of $K$.
By definition, $A_K$ does not have repeated factors. It is known that $A_K$ is always divisible by
$L-1$. The $A$-polynomial here is actually equal to $L-1$ times the $A$-polynomial defined in \cite{CCGLS}.

\subsection{The $B$-polynomial} It is also instructive to see the dual picture in the construction of the
$A$-polynomial. For an algebraic set $V$ (over $\BC$) let $\BC[V]$ denote the ring
of regular functions on $V$.  For example, $\BC[(\BC^*)^2/\tau]=
\ft^\sigma$, the $\sigma$-invariant subspace of  $\ft:=\BC[L^{\pm
1},M^{\pm 1}]$, where $\sigma(M^{k}L^{l})= M^{-k}L^{-l}$.

The map $\rho$ in the previous subsection induces an algebra
homomorphism
$$\theta: \BC[\chi(\partial X)]  \equiv \ft^\sigma \longrightarrow
\BC[\chi(X)].$$
We call the kernel  $\fp$ of $\theta$ the {\em classical
peripheral ideal}; it is an ideal of $\ft^\sigma$. We have the exact sequence
\be
0 \to \fp \to \ft^\sigma \overset{\theta}\longrightarrow \BC[\chi(X)].
\label{eq.orig0}
\ee

The ring $\ft^\sigma \subset \ft= \BC[M^{\pm1},L^{\pm 1}]$ embeds naturally into the
principal ideal domain $\tilde \ft:=\BC(M)[L^{\pm1}]$, where
$\BC(M)$ is the fractional field of $\BC[M]$. The ideal extension $\tilde{\fp}:=\tilde \ft \,\fp$
of $\fp$ in $\tilde \ft$ is thus generated by a single polynomial
$B_K\in \BZ[M,L]$ which has co-prime coefficients and is defined
up to a factor $\pm M^k$ with $k\in \BZ$. Again $B_K$ can be
chosen to have integer coefficients because everything can be
defined over $\BZ$. We call $B_K$ the \textit{$B$-polynomial} of $K$.
\label{B-polynomial}

\subsection{Relation between the $A$-polynomial and $B$-polynomial}

\label{Mfactor}

From the definitions one has immediately that the polynomial $B_K$ is $M$-essentially
divisible by $A_K$. Moreover, their zero sets $\{B_K=0\}$
and $\{A_K=0\}$ are equal, up to some
lines parallel to the $L$-axis in the $LM$-plane.


\begin{lemma} The field $\BC(M)$ is a flat $\BC[M^{\pm 1}]^\sigma$-algebra, and
 $\tilde{\ft} = \ft^{\sigma} \otimes _{\BC[M^{\pm 1}]^{\sigma}} \BC(M).$
\label{important}
\end{lemma}

\begin{proof} The extension from $\BC[M^{\pm 1}]^{\sigma}$ to $\BC(M)$ can be done in two steps: The first one is from
$\BC[M^{\pm 1}]^{\sigma}$ to $\BC[M^{\pm 1}]$ (note that $\BC[M^{\pm 1}]$ is free over $\BC[M^{\pm 1}]^{\sigma}$ since $\BC[M^{\pm 1}]=\BC[M^{\pm 1}]^{\sigma} \oplus M\BC[M^{\pm 1}]^{\sigma}$); the second step is from $\BC[M^{\pm 1}]$ to its field of fractions $\BC(M)$. Each step is a flat extension, hence $\BC(M)$ is flat over $\BC[M^{\pm 1}]^{\sigma}.$

It follows that the extension $(\ft^\sigma \hookrightarrow \ft)\otimes \BC(M)$ is still an injection, i.e.
\be
 \psi: \ft^\sigma \otimes _ {\BC[M^{\pm 1}]^\sigma} \BC(M) \to \ft\otimes _{ \BC[M^{\pm 1}] }  \BC(M) = \tilde\ft, \quad \psi(x\otimes y)= xy,
\notag
\ee
is  injective. Let us show that $\psi$ is surjective.
For every $n\in \BZ$,
\begin{eqnarray*}
L^n &=&  \psi \left ((ML^n+M^{-1}L^{-n}) \otimes  \frac{1}{M-M^{-1}}  -  (L^n+L^{-n}) \otimes \frac{M^{-1}}{M-M^{-1}}\right ).
\end{eqnarray*} Since $\{L^n\}_{n \in \BZ}$ generates $\tilde{\ft}=\BC(M)[L^{\pm 1}]$, $\psi$ is surjective.
Thus $\psi$ is an isomorphism.
\end{proof}

Consider the exact sequence \eqref{eq.orig0}. The ring $\BC[\chi(X)]$ has a $\ft^{\sigma}$-module structure via the algebra homomorphism $\theta: \BC[\chi(\partial X)]  \equiv \ft^{\sigma} \to \BC[\chi(X)]$, hence a
$\BC[M^{\pm 1}]^{\sigma}$-module structure since  $\BC[M^{\pm 1}]^{\sigma}$ is a subring of $\ft^{\sigma}$. By Lemma \ref{important}, $ \tilde{\ft} = \ft^{\sigma} \otimes _{\BC[M^{\pm 1}]^{\sigma}} \BC(M).$ It follows that $\tilde \fp = \fp \otimes _{\BC[M^{\pm 1}]^{\sigma}} \BC(M).$ Hence by taking the tensor product over $\BC[M^{\pm 1}]^\sigma$ of the exact sequence \eqref{eq.orig0} with $\BC(M)$, we get the exact sequence
\be
0 \to \bp \to \bt \overset{\btheta}\longrightarrow \widetilde{\BC[\chi(X)]},
\label{eq.orig}
\ee
where $\widetilde{\BC[\chi(X)]} :=  {\BC[\chi(X)]}\otimes_{\BC[M^{\pm 1}]^\sigma} \BC(M)$.
\begin{proposition}
The $B$-polynomial $B_K$ does not have repeated factors.
\label{B}
\end{proposition}

\begin{proof}

 We want to show that $\tilde{\fp}$ is radical, i.e. $\sqrt{\tilde{\fp}}=\tilde{\fp}.$ Here $\sqrt{\tilde{\fp}}$ denotes the radical of $\tilde{\fp}$.

Let $x:=M+M^{-1}$ and $$\underline{\ft} := \ft^{\sigma} \otimes _{\BC[M^{\pm 1}]^{\sigma}} \BC(x), \quad \underline{\fp} := \fp \otimes _{\BC[M^{\pm 1}]^{\sigma}} \BC(x).$$
Note that $\fp$, the kernel of $\theta: \ft^{\sigma} \to \BC[\chi(X)]$, is radical since the ring $\BC[\chi(X)]$ is reduced. We claim that $\underline{\fp}$ is also radical. Indeed, suppose $\gamma \in \underline{\ft}$ and $\gamma^2 \in \underline{\fp}.$ Then $\gamma^2=\delta/f$ for some $\delta \in \fp$ and $f \in \BC[x].$ It follows that $(f\gamma)^2=f\delta$ is in $\fp$. Hence $f\gamma \in \sqrt{\fp}=\fp$ which means $\gamma \in \underline{\fp}.$

Since $\underline{\ft}=\BC(x)[L^{\pm 1}]$ is a principal ideal domain, the radical ideal $\underline{\fp}$ can be generated by one element, say $\gamma(L) \in \BC(x)[L^{\pm 1}]$, which does not have repeated factors. Note that the polynomial $\gamma(L)$ and $\delta(L):=\gamma'(L)$, the derivative of $\gamma(L)$ with respect to $L$, are co-prime. Since $\BC(x)[L^{\pm 1}]$ is an Euclidean domain, there are $f, g \in \BC(x)$ such that $f\gamma+g\delta=1.$ It follows that $\gamma(L)$ and $\delta(L)$ are also co-prime in $\BC(M)[L^{\pm 1}].$ Hence the ideal $\tilde{\fp}=\underline{\fp} \otimes_{\BC(x)} \BC(M)$ in $\BC(M)[L^{\pm 1}]$ is radical. This means that the $B$-polynomial $B_K$ does not have repeated factors. \end{proof}

\begin{corollary}
For every knot $K$ one has
$$B_K=\frac{A_K}{\text{\em{its $M$-factor}}}.$$
\label{equal}
\end{corollary}

Here the $M$-factor of $A_K$ is the maximal factor of $A_K$ depending on $M$ only; it is defined up to a non-zero complex number.

\subsection{Small knots}

A knot $K$ is called \textit{small} if its complement $X$ does not contain closed essential surfaces. It is known that all two-bridge knots and all three-tangle pretzel knots are small \cite{HT, Oe}.

\begin{proposition}
Suppose $K$ is  a small knot. Then the $A$-polynomial $A_K$ has trivial $M$-factor. Hence the $A$-polynomial and $B$-polynomial of a small knot are equal.
\label{small}
\end{proposition}

\begin{proof}
The $A$-polynomial $A_K$ always contains the factor $L-1$ coming from characters of abelian representations \cite{CCGLS}. Hence we write $A_K=(L-1)A_{\emph nab}$ where $A_{\emph nab}$ is a polynomial in $\BC[M,L].$

Suppose the polynomial $A_{\emph nab}$ of a knot has non-trivial $M$-factor, then the Newton polygon of $A_{\emph nab}$ has the slope infinity. It is known that every slope of the Newton polygon of $A_{\emph nab}$ is a boundary slope of the knot complement \cite{CCGLS}. Hence the knot complement has boundary slope infinity. The complement of a small knot in $S^3$ does not have boundary slope infinity (this fact follows easily from \cite[Theorem 2.0.3]{CGLS}), hence its polynomial $A_{\emph nab}$ has trivial $M$-factor.
\end{proof}

\begin{remark}
By \cite{Mattman}, according to a calculation by Culler, there exists a non-small knot whose $A$-polynomial has non-trivial $M$-factor; it is the knot $9_{38}$ in the Rolfsen table.
\end{remark}

\section{Skein modules and the AJ conjecture}

Our proofs of the main theorems are based on the
fact that the KBSM is a quantization of the $SL_2(\BC)$-character variety \cite{Bul,PS} which has
been exploited in the work of Frohman, Gelca and Lofaro
\cite{FGL} where they defined the non-commutative $A$-ideal. In this section we will discuss the role of localized skein modules in our approach to the  AJ conjecture, and then prove Theorems \ref{t1} and \ref{t2}.

\label{AJ}

\def\CC{\mathcal C}
\subsection{Ring extensions} Suppose $R_1$ is a ring (with unit), $\CC$ is an $R_1$-complex, and $R_2$ is an $R_1$-algebra. We will say that $R_2 \otimes_{R_1} \CC$ is obtained from $\CC$ by a change
of ground ring.

Recall that $\CR=\BC[t^{\pm 1}]$.
We often consider $\BC$ as an $\CR$-algebra by $\BC \equiv \CR/((1+t))$. In this case, we use the notation $\ve(\CC):= \CC\otimes_\CR \BC$, where $\CC$ is an $\CR$-complex or an $\CR$-module.
Thus if $\CM$ is an $\CR$-module, then
$$ \ve(\CM) = \CM/((1+t)\CM).$$
If $\hat \CR$ is an $\CR$-algebra and $\CM$ is an $\hat \CR$-module, then one can easily see that
$$
\ve(\CM) = \CM \otimes_{\hat \CR} \ve(\hat \CR).
$$

\subsection{Skein modules as quantizations of character varieties}
\label{quan}
\def\CN{\mathcal N}

\no{
Recall that $\CR=\BC[t^{\pm1}]$.
For $\CR$-modules $\CM, \CN$ and $\CR$-morphism $f: \CM \to \CN$ let
$$\ve(\CM)= \CM/(1+t)\CM= \CM \otimes_\CR \BC,$$ where $\BC$ is considered as a $\CR$-module by the reduction $t=-1$, and
$\ve(f): \ve(\CM) \to \ve(\CN)$ be the corresponding $\BC$-morphism. If $\hat \CR$ is an $\CR$-algebra and $\CM$ is an $\hat \CR$-module, then one can easily see that
\be
\ve(\CM) = \CM \otimes_{\hat \CR} \ve(\hat \CR), \label{eq.17}.
\ee
}
\def\fs{\mathfrak s}
\def\univ{\mathrm{univ}}
 An important result \cite{Bul,PS} in the theory of skein modules is that
$\fs(Y):=\varepsilon (\CS(Y))$, the skein module at $t=-1$, has a natural $\BC$-algebra structure and is isomorphic to the universal $SL_2$-character algebra $\BC^\univ[\chi(Y)]$ of $\pi_1(Y)$. The product of two links in $\fs(Y)$ is their disjoint union, which  is well-defined when $t=-1$. The isomorphism between
$\fs(Y)$ and the universal $SL_2$-character algebra of $\pi_1(Y)$ is given by $K(r)= -\text{tr}\, r(K)$,
where $K$ is a knot in $Y$ representing an
element of $\pi_1(Y)$ and $r:\pi_1(Y) \to
SL_2(\BC)$ is a representation of $\pi_1(Y)$.
The quotient of $\fs(Y)$
 by its nilradical is canonically isomorphic to
$\BC[\chi(Y)]$, the ring of regular functions on the $SL_2$-character
variety of $\pi_1(Y)$.

For the case when $\CS=\CS(X)$, where $X$ is the knot complement, we have
\be
\ve(\CT^\sigma \overset{\Theta}{\longrightarrow} \CS)= \ (\ft^\sigma \overset{\theta}{\longrightarrow} \fs),
\label{eq.00}
\ee
where $\fs= \fs(X)\ =  \BC^\univ[\chi(X)]$.

In many cases $\fs(Y)$ is reduced, i.e. its nilradical is zero,
and hence $\fs(Y)$ is exactly the ring of
regular functions on the $SL_2$-character variety of $\pi_1(Y)$. For
example, this is the case when $Y$ is a torus, or when $Y$ is the
complement of a two-bridge knot/link \cite{Le06, PS, LT}, or when $Y$ is the complement of the $(-2,3,2n + 1)$-pretzel knot for any integer $n$ (see Section \ref{pretzel} below). We conjecture that

\begin{conjecture} For every knot $K$ in $S^3$, the universal $SL_2$-character ring of $\pi_1(S^3 \setminus K)$ is reduced.
\label{c3}
\end{conjecture}

\subsection{The non-reduced kernel}
Extending the right hand side of \eqref{eq.00} from the ground ring ${\BC[M^{\pm 1}]}^{\sigma}$ to $\BC(M)$, we get
\be
(\bbt \overset{\bbtheta}{\longrightarrow} \bbs) :=(\ft^\sigma \overset{\theta}{\longrightarrow} \fs)\otimes _{\BC[M^{\pm 1}]^{\sigma}} \BC(M).
\label{eq.00a}
\ee
We call the $\BC(M)$-vector space $\bbs$ the \textit{localized universal character ring} of $\pi_1(X)$.

Let $\bbp:= \ker \bbtheta \subset \bbt= \BC(M)[L^{\pm1}]$. We have the commutative diagram with exact rows
$$
\begin{CD}
0 @>>> \bbp @>>> \bbt @>\bbtheta >> \bbs \\
@. @VVV @| @V q VV \\
0 @>>>\bp @>>> \bt  @>\btheta >> \bbs/\sqrt0
\end{CD}
$$
where $q$ is the quotient map. Note that the second row of the above diagram is exactly the sequence \eqref{eq.orig}.

Both $\bbp$ and $\bp$ are ideals in the principal ideal domain
$\bbt=\bt$. Recall that $B_K$ is a generator of $\bp$. Let $\ov B_K$ be a generator of  $\bbp$. 
\begin{lemma}
One has $B_K \mid \ov B_K \mid (B_K)^l$ for some positive integer $l$. Consequently, $\bbt/\bbp$ is a finite dimensional $\BC(M)$-vector space.
\label{finite}
\end{lemma}
\begin{proof}
 Since $\bbtheta(\ov B_K)=0$, one has $\btheta(\ov B_K)=0$. This implies $\ov B_K \in \bp$, and hence $B_K \mid \ov B_K$.

 Since $\btheta(B_K)=0$, one has $\bbtheta(B_K) \in \sqrt 0$. It follows that $(B_K)^l \in \ov{\fp}$ for some positive integer $l$, and hence $\ov B_K \mid (B_K)^l$.

 Note that $B_K \neq 0$ (since $A_K \neq 0$), hence we also have $\ov B_K \neq 0$. If  $\ov B_K=L^d + \sum_{j=0}^{d-1} a_j(M) L^j$, with $a_j(M) \in \BC(M)$ and $d \ge0$, then the dimension of the $\BC(M)$-vector space $\bbt/\bbp$ is $d$.
\end{proof}

\subsection{Localization}

Let $D:=\CR [M^{\pm 1}]=\BC [t^{\pm1},M^{\pm 1}]$ and $\ov{D}$ be its localization at $(1+t)$:
$$\ov{D}:=\left\{\frac{f}{g} \mid f,g \in D, \, g \not \in (1+t)   D     
\right\}
,$$
which is a discrete valuation ring and  is flat over $D$.

The ring $D=\CR [M^{\pm 1}]$ is flat over $D^\sigma=\CR [M^{\pm 1}]^\sigma$, where $\sigma(M)=M^{-1}$, since it is free over $\CR [M^{\pm 1}]^\sigma$:
$$ \CR [M^{\pm 1}]=  \CR [M^{\pm 1}]^\sigma \oplus M \, \CR [M^{\pm 1}]^\sigma .$$

The quantum torus $\CT$ is a $D^\sigma$-algebra. Let
$ \ov{\CT}:=\CT \otimes_{D^\sigma} \ov{D}$. Similar to Lemma \ref{important}, we have
$$\ov{\CT} =\left\{\sum_{j \in \BZ} a_j(M)L^j \mid a_j(M) \in \ov{D},~a_j=0\text{~almost everywhere}\right\},$$
with commutation rule $a(M)L^{k} \cdot b(M)L^{l}=a(M)b(t^{2k}M)L^{k+l}$.

\subsection{The localized skein module} Let $\CS:=\CS(X)$ be the skein module of the knot complement $X$.
\begin{defn}
The \textit{localized skein module} of the knot complement $X$ is the $\bD$-module $\bS := \CS \otimes_{D^\sigma} \bD$.
We say that $\bS$ is finitely generated if it is finitely generated as a $\bD$-module.
\end{defn}

Recall from Subsection \ref{OP} that we have the map $\Theta: \CT^\sigma \to \CS$, which is considered as a
$D^\sigma$-morphism. Let $\bTheta: = \Theta \otimes_{D^\sigma} \bD$, i.e.
$$
( \bT \overset{\bTheta}\longrightarrow  \bS) \quad = \quad \big( \CT^\sigma \overset{\Theta}\longrightarrow \CS \big) \otimes_{D^\sigma} \bD.
$$

\def\IM{\mathrm{Im}}

\begin{lemma}
\label{assumption}
One has
$$
\ve( \bT \overset{\bTheta} \longrightarrow \bS)= (\bbt \overset{\bbtheta} \longrightarrow \bbs).
$$
\end{lemma}

\begin{proof} Note that $\BC(M)$ is a $D^\sigma$-algebra by the composition of two maps
\be
D^\sigma \hookrightarrow \bD \to \ve(\bD)= \BC(M).
\notag 
\ee
Hence $\otimes _{D^\sigma} \BC(M)$ is the composition of two tensor products
\be
 (\CT^\sigma \overset{\Theta}\longrightarrow \CS) \otimes _{D^\sigma} \BC(M)= \ve \big(  (\CT^\sigma \overset{\Theta}\longrightarrow \CS) \otimes_{D^\sigma} \bD    \big) = \ve (\bT \overset{\bTheta}\longrightarrow  \bS)
 \label{eq.30a}.
\ee

The same $D^\sigma$-algebra structure of $\BC(M)$  can also be obtained by the composition of
\be
D^\sigma  \to \ve (D^\sigma)= \BC[M^{\pm1}]^\sigma   \hookrightarrow  \BC(M).
\notag 
\ee
Hence the  left hand side of \eqref{eq.30a} can be written as
 \be
 (\CT^\sigma \overset{\Theta}\longrightarrow \CS) \otimes _{D^\sigma} \BC(M) = \big( \ve ( \CT^\sigma \overset{\Theta}\longrightarrow \CS )\big) \otimes_{\BC[M^{\pm1}]^\sigma}\BC(M) = (\bbt \overset{\bbtheta} \longrightarrow \bbs),
 \label{eq.30}  \ee
where the last identity follows from the definitions \eqref{eq.00} and \eqref{eq.00a}.

The lemma follows from  \eqref{eq.30a} and \eqref{eq.30}. \end{proof}

\subsection{Left ideals of $\bT$} Recall that $\bT =  \bD[L^{\pm 1}]$ is the set of all Laurent polynomials
$$ \sum_{j \in \BZ} a_j(M)L^j, \quad    a_j(M) \in \ov{D} \text{ and } a_j=0  \text { for almost every }j \in \BZ.$$
Let $\bT_+$ be the subring of $\bT$ consisting of all polynomials in $L$, i.e. polynomials like the above  with $a_j(M)=0$ if $j <0$.
For $f, \, g$ in $\bT_+$, we say that $f$ is divisible by $g$ and write $g \mid f$ if there exists $h \in \bT_+$  such that $f=hg$. 

Although the ring $\bT$ is not a left PID, we have the following description of its ideals.

\begin{proposition} \label{lem.ideal}

Suppose $I\subset \bT$ is a non-zero left ideal.
There are $h_0,\dots, h_{m-1} \in \bT_+$ with leading coefficients 1 and $\gamma \in \bT_+$ 
such that
$I$ is generated by $\{ h_0 \gamma, (1+t) h_{1} \, \gamma, \dots, (1+t)^{m-1} h_{m-1}  \gamma, (1+t)^m \gamma\}$.
Besides, $1\le \deg_L(h_{j+1}) \le  \deg_L(h_j)$ for $j= 0,\dots, m-2$; and $\gamma$ is the generator of the principal left ideal $\tilde I=I \cdot \tilde \CT \subset \tilde \CT$.

\end{proposition}

\begin{proof} Note that any ideal  of $\bD$  is a power of the prime ideal $(1+t)$. Suppose $\CM$ is a $\bD$-module. We say that $u\in \CM$ has {\em height} $k\in \BZ_{\ge 0}$ if
$$u\in (1+t)^k \CM \setminus (1+t)^{k+1} \CM.$$

We have the following {\em weak division algorithm}: Suppose $f,g \in \bT_+$, with $g\neq0$. Assume that  $\deg (f) \ge \deg (g)$, and  the height of the leading coefficient of $f$ is
 greater than or equal to that of the leading coefficient of $g$.
Then there are $q, r\in \bD[L]$ such that
$$ f = q g + r,$$
where $\deg(r) < \deg(f)$. Here we use the convention that $\deg(0)= -\infty$. In fact, one can take $q= a\, L^k$, where $a$ is the quotient of the leading of $f$ by that of $g$, and
$k$ is the difference between the degree of $f$ and that of $g$.

We will frequently use the weak division algorithm with $f,g \in I$. Then the remainder $r$ is also in $I$.

Let $I_+= I \cap \bT_+$, and
$I_n$ be the set of all elements in $I_+\setminus \{0\}$ whose leading coefficient has height $n$. By definition
$$
I_+\setminus \{0\} = \sqcup_{n=0}^\infty \, I_n.
$$
Choose the smallest integer $N \ge 0$ such that $I_N \not= \emptyset$.

\smallskip

\underline{Claim 1}:  $I_n \not= \emptyset$ for every $n \ge N$.

\begin{proof} (of Claim 1)
Since $(1+t)^{n-N} I_N \subset I_n$ for every $n \ge N$, each $I_n \neq \emptyset$.
\end{proof}

Suppose $d_n$ is the least $L$-degree of elements of $I_n$, and choose $f_n \in I_n$ such that the degree of $f_n$ is $d_n$. The choice of $f_n$ guarantees that if $f\in I_n$ than one can divide $f$ by $f_n$ using the
weak division algorithm.

Since $(1+t) I_n \subset I_{n+1}$, we have  $d_N \ge d_{N+1} \ge d_{N+2} \ge \dots  $. Hence the sequence of decreasing non-negative integers $d_N, d_{N+1},\dots$ eventually stabilizes. Let $m \ge 0$ be the smallest integer such that $d_{N+m}= d_{N+m+j}$ for every $j =0,1,2, \dots$.

\smallskip

\underline{Claim 2}: If $f \in I_+\setminus \{0\}$  has degree $< d_j$, then $f \in I_{j'}$ for some $j' >j$.

\begin{proof} (of Claim 2) Since $I_+\setminus \{0\} = \sqcup_{n=0}^\infty \, I_n$, $f \in I_{j'}$ for some $j'$. If $j ' \le j$ then $\deg(f) \ge d_{j'} \ge d_j$, a contradiction. Hence  $j' > j$.
\end{proof}

\underline{Claim 3}: If $f \in I_+  \setminus \{0\}$ has degree $\le d_{N+m}$, then $f$ is divisible by $f_{N+m}$.

\begin{proof} (of Claim 3) Suppose $f \in I_+  \setminus \{0\}$ has degree $\le d_{N+m}$. Since $\deg(f)<d_{N+m-1}$, by Claim 2, $f \in I_j$ for some $j \ge N+m$. Note that $\deg(f) \ge d_j=d_{N+m}$. Dividing $f$ by $f_{N+m}$ using the weak division algorithm, the remainder $r$ has degree $<\deg(f)=d_{N+m}$. Since there are no elements in  $I_+  \setminus \{0\}$ of degree $< d_{N+m}$, we must have $r=0$, which implies that $f$ is divisible by $f_{N+m}$.
\end{proof}

 For $0 \le j \le m$ let $I^{(j)}$ be the left ideal of $\bT_+$ generated by $f_{N+j}, f_{N+j+1}, \cdots, f_{N+m}$.

\smallskip

\underline{Claim 4}: Suppose $f \in I_+\setminus \{0\} $ has degree $< d_{N+j}$, where $0 \le j < m$, then $f \in I^{(j+1)}$.

\begin{proof} (of Claim 4) We use induction on the degree of $f$.  Suppose $f \in I_+\setminus \{0\}$  has degree $< d_{N+j}$ for $0 \le j <m$. Then, by Claim 2, $f \in I_{j'}$ for some $j' \ge j+1$. Dividing $f$ by $f_{N+j'}$ using the weak division algorithm, the remainder has degree $<\deg(f)<d_{N+j}$, hence, by induction,
 it belongs to $I^{(j+1)}$.

If $j'>m$ then, by Claim 3, $f_{j'}$ is divisible by $f_{N+m}$. Otherwise, i.e. if $j' \le m$, then $f_{j'}$ belongs to $I^{(j+1)}$. Hence we always have $f_{j'} \in I^{(j+1)}$. It follows that $f \in I^{(j+1)}$.
\end{proof}

\underline{Claim 5}: $I_+= I^{(0)}$, i.e. $I_+$ is generated by $\{ f_j\mid j=N,\dots, N+m\}$.

\begin{proof} (of Claim 5) We use induction on the degree of $f \in I_+$. If the degree of $f < d_N$, then $f\in I^{(1)}$ by Claim 4. Suppose the degree of $f$ is $\ge d_N$. Dividing $f$ by $f_N$ using the weak division algorithm, the remainder has degree  less than
   that of $f$ and hence belongs to $I^{(0)}$ by induction hypothesis. Thus $f \in I^{(0)}$.
\end{proof}

\underline{Claim 6}:  For every $0 \le j \le m$, $f_{N+m}$ divides $(1+t)^{m-j} f_{N+j}$.

\begin{proof} (of Claim 6)  We use induction,  beginning with the case $j=m$ which is obvious. Suppose $j \le m-1$. Dividing $(1+t)f_{N+j}$ by $f_{N+j+1}$ using the weak division algorithm, the remainder $r$ is an element in $I_+$ of degree $< d_{N+j}$. By Claim 4, $r$ is an element in $I^{(j+1)}=( f_{N+j+1}, \cdots, f_{N+m}) \subset I_+$. It follows that $(1+t)f_{N+j}$ is an element in $I^{(j+1)}$.  By induction hypothesis, every element in $(1+t)^{m-j-1}I^{(j+1)}$ is divisible by $f_{N+m}$. In particular, $(1+t)^{m-j}f_{N+j}$ is divisible by $f_{N+m}$.
\end{proof}

\textit{End of Proof of Proposition \ref{lem.ideal}.} By Claim 6, $(1+t)^m f_N = h_0 f_{N+m}$, for some $h_0\in \bT_+$. Comparing the leading coefficients, we see that the leading coefficient of $h_0$ is 1.
  From $(1+t)^{m}f_N = h_0 f_{N+m}$, with $h_0$ having leading coefficient 1, one can easily show that $f_{N+m}$ is divisible by $(1+t)^m$. Hence $f_{N+m} = (1+t)^m \gamma$, where $\gamma\in \bT_+$.

By Claim 6, for each $0 \le j \le m-1$ there is $h_j \in \bT_+$ (whose leading coefficient is 1) such that $(1+t)^{m-j} f_{N+j}= h_j f_{N+m} = (1+t)^m h_j \gamma$, i.e. $f_{N+j}=(1+t)^j h_j \gamma$.
  \end{proof}

\subsection{Assumption $\bbtheta$ is surjective and $\bS$ is finitely generated}

\begin{lemma}\label{lem.surj}
Suppose $\bbtheta$ is surjective and $\bS$ is finitely generated. Then $\bTheta$ is surjective.
\end{lemma}

\begin{proof}
From Lemma \ref{assumption}, one has the commutative diagram
\be \begin{CD}  \bT   @>\bTheta  >> \bS \\
  @V \varepsilon VV  @V\ve VV \\
  \bbt     @>\bbtheta >> \bbs
\end{CD}
\label{dia.1a}
\ee

Suppose $\{x_1,\dots, x_d \}$ is a basis of the $\BC(M)$-vector space $\bbs$. Let $\ov{x}_j \in \bS$ be a lift of $x_j$. By Nakayama's Lemma, $\{\ov x_1,\dots, \ov x_d \}$ spans $\bS$ over $\bD$.
Since $\bbtheta$ and $\ve$ in diagram \eqref{dia.1a} are surjective, each $\ov x_j$ is in the image of $\bTheta$. This proves  that $\bTheta$ is surjective.
\end{proof}

\begin{proposition}
\label{surj}
Suppose $\bbtheta$ is surjective and $\bS$ is finitely generated. Then $\ve(\beta_K) \mid \ov B_K$.
\end{proposition}

\begin{proof} Recall that $\bbp= \ker \bbtheta$.  By Lemma \ref{lem.surj}, $\bTheta$ is surjective.
Diagram~\eqref{dia.1a} can be extended to   the following commutative diagram
with exact rows:
$$ \begin{CD}  0 @ >>> \bP @ >\iota>> \bT   @>\bTheta >> \bS  @>>>  0\\
@.  @ V h VV @V \varepsilon VV  @V \varepsilon VV \\
0 @ >>>     \bbp  @ >>>  \bbt     @>\bbtheta >> \bbs @>>> 0
\end{CD}
$$

Taking the tensor product of the first row, which is an exact sequence of $\bD$-modules, with the $\bD$-algebra $\BC(M)$, we get the exact sequence
$$ \bP  \otimes_{\bD} \BC(M) \overset{\ve(\iota)} \longrightarrow  \bbt \overset{\bbtheta} \longrightarrow \bbs \to 0.$$
It follows that  $\bbp= \ker(\bbtheta)= \IM(\ve(\iota))= \IM(\ve \circ \iota)=   h(\bP)$.

Suppose $\{g_j:=(1+t)^j h_j \gamma, j=0,1,\dots,m\}$ (with $h_m=1$) be a set of generators of $I=\bP$ as described in  Proposition \ref{lem.ideal}.
Then $h(g_j)=\ve(g_j)=0$ except possibly for $j=0$. It follows that $\bbp= h(\bP)$ is the principal ideal generated by $\ve(g_0)$. Hence $\ve(g_0)=\ov B_K \neq 0$.
On the other hand, it is clear that $\beta_K= \gamma$. Hence  $\ve(\beta_K) \mid \ve(g_0)= \ov B_K$.
\end{proof}

\no{Suppose $\{(1+t)^j g_j, j=0,1,\dots,m\}$ be a set of generators of $I=\bP$ as described in  Proposition \ref{lem.ideal}.
Then $\ve((1+t)^j g_j)=0$ except for $j=0$. It follows that $\bbp= h(\bP)$ is the principal ideal generated by $\ve(g_0)$. Hence $\ve(g_0)=\ov B_K $.
On the other hand, by Proposition~\ref{lem.ideal}(b), $\al_K \mid g_0$, hence $\ve(\al_K) \mid \ve(g_0)= \ov B_K$.}

Note that if the localized universal character ring $\bbs$ is reduced, then $\ov B_K=B_K$. Since $\alpha_K \mid \beta_K$, Proposition \ref{surj} implies the following.

\begin{corollary}
Suppose $\bbtheta$ is surjective, $\bS$ is finitely generated, and $\bbs$ is reduced. Then $\ve(\al_K) \mid B_K$ in $\tilde{\ft}=\BC(M)[L^{\pm 1}]$.
\label{aB}
\end{corollary}

\def\wft{\tilde{\ft}}
\def\wtheta{\tilde{\theta}}

\subsection{Proofs of Theorems \ref{t1} and \ref{t2}}

\subsubsection{Proof of Theorem \ref{t1}}

Since $K$ is a hyperbolic knot, it has discrete faithful $SL_2(\BC)$-representations. Let $\chi_0$ be an irreducible component of $\chi(X)$ containing the character of a discrete faithful $SL_2(\BC)$-representation. Since $X$ has one boundary component, by a result of Thurston \cite{Th} $\chi_0$ has dimension 1.

\begin{lemma}\label{lem.0}
Suppose $R$ is a \ $\BC$-algebra which is an integral domain, and the transcendence degree of the fractional field $F(R)$ of $R$ over $\BC$ is 1. Suppose $x \in R$ is transcendental over $\BC$.
Then the natural  map $R \otimes_{\BC[x]} \BC(x) \to F(R)$ is an isomorphism.
\end{lemma}
\begin{proof} Note that $\BC(x)$ is flat over $\BC[x]$ and $R \otimes_{\BC[x]} \BC(x) \subset F(R)$. Hence we only need to show that $R \otimes_{\BC[x]} \BC(x)$ is a field, or that every $0 \neq y \in R$ is invertible in $R \otimes_{\BC[x]} \BC(x)$. Fix $y \in R$, $y \neq 0$. Since $x \in R$ is transcendental over $\BC$, the transcendence degree of the field $\BC(x)$ over $\BC$ is 1. The field $\BC(x)$ is contained in the fractional field $F(R)$ of $R$ whose transcendence degree over $\BC$ is also $1$, hence $F(R)$ is algebraic over $\BC(x)$. In particular, $y$ is algebraic over $\BC(x)$. Since $\BC(x)[y]$ is a subfield of $F(R)$, $y^{-1} \in \BC(x)[y]$. Clearly $\BC(x)[y] \subset R \otimes_{\BC[x]} \BC(x)$. Hence $y^{-1} \in R \otimes_{\BC[x]} \BC(x)$.
\end{proof}

Recall that the inclusion $\partial X \hookrightarrow X$ induces the restriction map $\rho : \chi(X) \to \chi(\partial X).$ Let $Y_0$ be the Zariski closure of
$\rho(\chi_0) \subset \chi(\partial X)$. Then $Y_0$ is irreducible and has dimension 1. Since $\rho |_{\chi_0}: \chi_0 \to Y_0$ has dense image, the pullback map $(\rho |_{\chi_0})^*: \BC[Y_0]\to \BC[\chi_0]$ is an embedding. Both $\BC[Y_0]$ and $\BC[\chi_0]$
are integral domains. \cite[Theorem 3.1]{Du} says that $(\rho |_{\chi_0})^*$ induces an isomorphism, also denoted by $(\rho |_{\chi_0})^*$, $\BC(Y_0) \to \BC(\chi_0)$, where $\BC(Y_0)$ (resp. $\BC(\chi_0)$) is the fractional field of $\BC[Y_0]$ (resp. $\BC[\chi_0]$).

Let $x\in \BC[Y_0]\subset \BC[\chi_0]$ be defined by $x(\rho)= \tr (\rho(\mu))=M+M^{-1}$. By \cite{CSx}, $x$ is not a constant function on $\chi_0$. It follows that $x$ is transcendental over $\BC$.
Hence Lemma \ref{lem.0} implies that $(\rho |_{\chi_0})^*:\BC[Y_0] \otimes_{\BC[x]} \BC(x) \to \BC[\chi_0] \otimes_{\BC[x]} \BC(x)$ is an isomorphism. 

By assumption, the character variety $\chi(X)$ consists of two irreducible components: the abelian component, denoted by $\chi_{ab}$, and $\chi_0$. It is known that $\chi_{ab}$ has dimension 1 and the restriction map $\rho|_{\chi_{ab}}:\chi_{ab} \to \chi(\partial X)$ is a birational isomorphism onto its image. It follows that $(\rho |_{\chi_{ab}})^*:\BC[Y_{ab}] \otimes_{\BC[x]} \BC(x) \to \BC[\chi_{ab}] \otimes_{\BC[x]} \BC(x)$ is an isomorphism, where $Y_{ab}$ is the Zariski closure of $\rho(\chi_{ab}) \subset \chi(\partial X)$. 

\begin{lemma}
One has
\begin{equation}
\BC[\chi(X)] \otimes_{\BC[x]} \BC(x) \cong \left(\BC[\chi_0] \otimes_{\BC[x]} \BC(x)\right) \times \left(\BC[\chi_{ab}] \otimes_{\BC[x]} \BC(x)\right).
\label{exactX}
\end{equation}
\label{points}
\end{lemma}

\begin{proof}
From Subsection \ref{universaldef}, we see that $\BC[\chi(X)]$ is the quotient of the polynomial ring $R:=\BC[x,y_1, \cdots, y_m]$ by an ideal $I \subset R$, where $x=M+M^{-1}$ is the trace of the meridian and $y_1, \dots, y_m$ are traces of some fix elements in the knot group.

Suppose $\BC[\chi_0]=R/I_0$ and $\BC[\chi_{ab}]=R/I_{ab}$ where $I_0,\,I_{ab}$ are ideals in $R$. Since $\chi(X)$ consists of two irreducible components $\chi_{ab}$ and $\chi_0$, $I=I_0 \cap I_{ab}$, and hence $\BC[\chi(X)]=R/(I_0 \cap I_{ab})$. Consider the following sequence of $\BC[x]$-modules
\be
0 \to R/(I_0 \cap I_{ab})  \overset{\eta} \longrightarrow  R/I_0 \times R/I_{ab}  \overset{\xi} \longrightarrow  R/(I_0+I_{ab}) \to 0.
\label{exact}
\ee
where $\eta(y)=(y,y)$ and $\xi(y,z)=y-z$. It is easy to check that the sequence \eqref{exact} is exact.

Consider the $\BC[x]$-module $R/(I_0+I_{ab})$. The zero set of the ideal $(I_0+I_{ab}) \subset R$ is precisely the intersection of the two algebraic varieties $\chi_0$ and $\chi_{ab}$, hence it is a set consisting of a finite number of points. It follows that $R/(I_0+I_{ab}) \cong \BC^k$ for some $k$. The exact sequence \eqref{exact} can be rewritten as
\be
0 \to R/(I_0 \cap I_{ab})  \overset{\eta} \longrightarrow  R/I_0 \times R/I_{ab}  \overset{\xi} \longrightarrow  \BC^k \to 0.
\label{exact1}
\ee
Since $\BC(x)$ is flat over $\BC[x]$ and $\BC^k \otimes_{\BC[x]} \BC(x)=0$, the following sequence which is obtained from \eqref{exact1} by tensoring it with $\BC(x)$ over $\BC[x]$ is exact
$$
0 \to R/(I_0 \cap I_{ab}) \otimes_{\BC[x]} \BC(x)  \overset{\eta \otimes_{\BC[x]} \BC(x)} \longrightarrow  (R/I_0 \times R/I_{ab} ) \otimes_{\BC[x]} \BC(x) \overset{\xi \otimes_{\BC[x]} \BC(x)} \longrightarrow  0 .
$$
The lemma follows.
\end{proof}

Let $Y \subset \chi(\partial X)$ be the algebraic set consisting of $Y_0$ and $Y_{ab}$. By similar arguments as in the proof of Lemma \ref{points}, we have
\begin{equation}
\BC[Y] \otimes_{\BC[x]} \BC(x) \cong \left(\BC[Y_0] \otimes_{\BC[x]} \BC(x)\right) \times \left(\BC[Y_{ab}] \otimes_{\BC[x]} \BC(x)\right).
\label{exactY}
\end{equation}
We have $$(\rho |_{\chi_0})^* \times (\rho |_{\chi_{ab}})^*:\left(\BC[Y_0] \otimes_{\BC[x]} \BC(x)\right) \times (\left(\BC[Y_{ab}] \otimes_{\BC[x]} \BC(x)\right)) $$
$$\longrightarrow \left(\BC[\chi_0] \otimes_{\BC[x]} \BC(x)\right) \times \left(\BC[\chi_{ab}] \otimes_{\BC[x]} \BC(x)\right)$$ is an isomorphism. Equations \eqref{exactX} and \eqref{exactY} then imply that  $(\rho |_{\chi_0})^* \times (\rho |_{\chi_{ab}})^*:\BC[Y] \otimes_{\BC[x]} \BC(x) \to \BC[\chi(X)] \otimes_{\BC[x]} \BC(x)$ is an isomorphism. Since $Y$ is an algebraic set in $\chi(\partial X)$, $\BC[Y]$ is a quotient of $\BC[\chi(\partial X)] \equiv \ft^{\sigma}$. Hence the map $\ft^{\sigma} \otimes_{\BC[x]} \BC(x) \to \BC[\chi(X)] \otimes_{\BC[x]} \BC(x)$, induced by $\rho : \chi(X) \to \chi(\partial X)$, is surjective. Taking the tensor product of this map with $\BC(M)$ over $\BC(x)$, we get the map
$$\tilde{\theta}: \tilde{\ft}=\ft^{\sigma} \otimes_{\BC[x]}\BC(M) \to \widetilde{\BC[\chi(X)]}=\BC[\chi(X)] \otimes_{\BC[x]} \BC(M), $$
which is also surjective. (Note that $\BC(M)$ is flat over $\BC(x)$.)

 Note that
$\tilde{\ft} = \BC(M)[L^{\pm 1}]$. Hence the dimension of the $\BC(M)$-vector space $\widetilde{\BC[\chi(X)]}$ is equal to the $L$-degree of $B_K$, the generator of the kernel $\tilde{\fp}$ of $\tilde{\theta}$, and is finite. Since the universal character ring $\mathfrak{s}(X)$ is reduced and $\bS$ is a finitely generated $\bD$-module, Corollary \ref{aB} implies that $B_K$ is $M$-essentially divisible by $\ve(\alpha_K)$. Since $A_K$ is $M$-essentially equal to $B_K$ by Corollary \ref{equal}, it must  be $M$-essentially divisible by $\ve(\alpha_K)$.

It is known that $A_K$ always contains the factor $L-1$ coming from characters of abelian representations \cite{CCGLS}), and $\varepsilon(\alpha_K)$ is also divisible by $L-1$ \cite[Proposition 2.3]{Le06}. Hence $\frac{ A_K}{L-1}$ is $M$-essentially divisible by $\frac{\varepsilon(\alpha_K)}{L-1}$.

Since the character variety $\chi(X)$ consists of two irreducible components, $A_K$ has exactly two irreducible factors. One factor is $L-1$, hence the other one, $\frac{A_K}{L-1}$, is irreducible. Since $\frac{ A_K}{L-1}$ is $M$-essentially divisible by $\frac{\varepsilon(\alpha_K)}{L-1}$, it follows that $$\frac{\varepsilon(\alpha_K)}{L-1} \eqM 1, \quad \text{or} \quad \frac{\varepsilon(\alpha_K)}{L-1} \eqM \frac{ A_K}{L-1}.$$ If $\frac{\varepsilon(\alpha_K)}{L-1} \eqM 1$, then, by Lemma \ref{prop17} below, the recurrence polynomial $\alpha_K$ has $L$-degree 1. This contradicts condition (iii) of Theorem 1, hence we must have $\frac{\varepsilon(\alpha_K)}{L-1} \eqM \frac{ A_K}{L-1}$. In other words, the AJ conjecture holds true for $K$.

\begin{lemma}
	The polynomial $\varepsilon(\alpha_K)$ is $M$-essentially equal to $L-1$ if and only if
	the $L$-degree of the recurrence polynomial $\alpha_K$ is $1$.
	\label{prop17}
\end{lemma}

\begin{proof}
	The backward direction is obvious since $\varepsilon(\alpha_K)$ is always divisible by $L-1$.
	
	Now suppose the polynomial $\varepsilon(\alpha_K)$ is $M$-essentially equal to $L-1$,
	i.e. $\varepsilon(\alpha_K)=g(M)(L-1)$ for some non-zero $g(M) \in \BC[M^{\pm 1}].$	Then
	\begin{equation}
		\alpha_K= g(M)(L-1)+(1+t) \sum_{j=0}^{d} a_j(M)L^j
		\label{eq23}
	\end{equation}
	where $a_j(M)$'s are Laurent polynomials in $\CR[M^{\pm 1}]$ and $d$ is the $L$-degree of $\alpha_K.$
	
	By a result in \cite{Ga08}, the recurrence ideal $\CA_K$ is invariant under the involution $\sigma$. Hence $\sigma(\alpha_K)$ is an element in $\CA_K$. Since $\alpha_K$ is the generator of $\wt{\CA}_K$, it follows that $\alpha_K=h(M)\sigma (\alpha_K)L^{d}$ for some $h(M) \in \CR(M)$. Equation \eqref{eq23} implies that
	$$
		g(M)(L-1)+(1+t) \sum_{j=0}^{d}a_j(M)L^j $$ $$= h(M)g(M^{-1})(L^{-1}-1)L^{d}+(1+t) \sum_{j=0}^{d} h(M)a_j(M^{-1})L^{d-j}.$$

	If $d>1$ then by comparing the coefficients of $L^0$ in both sides of the above equation, we get $-g(M)+(1+t)a_0(M)=(1+t)h(M)a_{d}(M^{-1}),$ i.e.
	\begin{equation}
	g(M)=(1+t)\left( a_0(M)-h(M)a_{d}(M^{-1}) \right)
	\label{eq24}
	\end{equation}
	 Since $g(M)$ is a Laurent polynomial in $M$ with coefficients in $\BC$, equation \eqref{eq24} implies that $g(M)$ must be equal to 0. This is a contradiction. Hence we must have $d=1$.
\end{proof}

\begin{remark}
In the proof of Theorem \ref{t1}, instead of condition {\em (iii)} we actually use the following weakened version 

{\em (iii') the localized universal $SL_2$-character ring $\bbs$ of $\pi_1(S^3 \setminus K)$ is reduced.}\\
(Recall that  $\bbs= \fs\otimes _{\BC[M+ M^{-1}]}\BC(M)$ is the localization of the universal $SL_2$-character ring.) Thus Theorem \ref{t1} holds if the condition {\em (iii)} is replaced by {\em (iii')} above.
\end{remark}

\subsubsection{Proof of Theorem \ref{t2}}

It is known that two-bridge knots and $(-2,3,2n+1)$-pretzel knots, excluding torus knots, are hyperbolic. (Note that the AJ conjecture holds true for torus knots by \cite{Hi, Tr}). Their universal character rings are reduced by \cite[Corollary 5.8]{Le06} and Theorem \ref{reduced} respectively. Their localized skein modules are finitely generated over the local ring $\bar{D}$ by \cite[Theorem 2]{Le06} and Proposition \ref{localized} respectively.

The $L$-degree of the recurrence polynomial of a two-bridge knot is $>1$ according to \cite[Proposition 2.2]{Le06}. Also by \cite[Proposition 2.2]{Le06}, for any knot if the $L$-degree of its recurrence polynomial is 1 then the breadth of its colored Jones polynomial is a linear function (in the color). For the $(-2,3,2n+1)$-pretzel knot, by \cite[Section 4.7]{Ga10} the breadth of its colored Jones polynomial is not a linear function, hence the $L$-degree of its recurrence polynomial is $>1$.

Double twist knots
of the form $J(k,l)$ with $k \not= l$, two-bridge knots of the form $\fb(p,m)$ with $m=3$ or ``$p$ prime and $\gcd(\frac{p-1}{2},\frac{m-1}{2})=1$'', and $(-2,3,6n \pm 1)$-pretzel knots satisfy condition (ii) of Theorem \ref{t1} by \cite{MPL}, Theorem \ref{prime} and \cite{Bur}, and \cite{Mat} respectively. Hence the theorem follows.

\section{The universal character ring of the $(-2,3,2n+1)$-pretzel knot}

\label{pretzel}

In this section we explicitly calculate the universal character ring of the $(-2,3,2n+1)$-pretzel knot and prove its reducedness for all integers $n$.

\subsection{The character variety} For the $(-2,3,2n+1)$-pretzel knot $K_{2n+1}$, we have
$$\pi_1(X)=\la a,b,c \mid cacb=acba,~ba(cb)^n=a(cb)^nc\ra,$$
where $X=S^3 \setminus K_{2n+1}$ and $a,b,c$ are meridians depicted in Figure $1$.

\begin{figure}[htpb]
$$ \psdraw{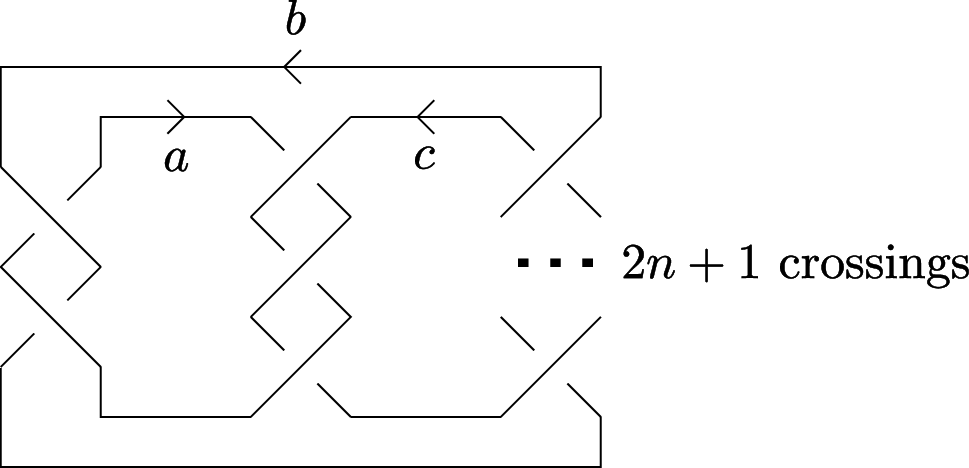}{3.5in} $$
\caption{The $(-2,3,2n+1)$-pretzel knot}
\end{figure}

Let $w=cb$ then the first relation of $\pi_1(X)$ becomes $caw=awa.$ It implies that $c=awaw^{-1}a^{-1}$ and  $b=c^{-1}w=awa^{-1}w^{-1}a^{-1}w$. The second relation then has the form
$$awa^{-1}w^{-1}a^{-1}waw^n=aw^nawaw^{-1}a^{-1}$$
i.e.
$$w^nawa^{-1}w^{-1}a^{-1}=a^{-1}w^{-1}awaw^{-1}w^n.$$
Hence we obtain a presentation of $\pi_1(X)$ with two generators and one relation
$$\pi_1(X) = \la a,w \mid w^nE=Fw^n\ra$$
where $E:=awa^{-1}w^{-1}a^{-1}$ and $F:=a^{-1}w^{-1}awaw^{-1}.$

The character variety of the free group $F_2=\la a,w \ra$ in 2 letters $a$ and $w$ is isomorphic to $\BC^3$ by the Fricke-Klein-Vogt theorem, see \cite{LM}. For every element $\omega \in F_2$ there is a unique polynomial $\mathbf{P}_\omega$ in 3 variables such that for any representation $r: F_2 \to SL_2(\BC)$ we have $\tr (r(\omega))=\mathbf{P}_\omega (x,y,z)$ where $x:=\tr(r(a)),~y:=\tr(r(w))$ and $z:=\tr(r(aw))$. The polynomial $\mathbf{P}_\omega$ can be calculated inductively using the following identities for traces of matrices $A,B \in SL_2(\BC)$:
\begin{equation}
\tr(A)=\tr(A^{-1}), \quad \tr(AB)+\tr(AB^{-1})=\tr(A)\tr(B).
\label{41}
\end{equation}
Thus for every representation $r: \pi_1(X) \to SL_2(\BC)$, we consider $x,y$, and $z$ as functions of $r$. The character variety of $\pi_1(X)$ is the zero locus of an ideal in $\BC[x,y,z]$, which we describe explicitly in the next theorem.

\begin{theorem}
The character variety of the pretzel knot $K_{2n+1}$ is the zero locus of 2 polynomials $P:=\mathbf{P}_E-\mathbf{P}_F$ and $Q_n:=\mathbf{P}_{w^nEa}-\mathbf{P}_{Fw^na}$. Explicitly,
\begin{eqnarray}
P &=& x - x y + (-3 + x^2 + y^2) z - x y z^2 + z^3,\label{42}\\
Q_n&=& S_{n-2}(y)+S_{n-3}(y)-S_{n-4}(y)-S_{n-5}(y)-S_{n-2}(y) \, x^2 \label{43}\\
   && + \, \big( S_{n-1}(y)+S_{n-3}(y)+S_{n-4}(y) \big) \, xz-\big( S_{n-2}(y)+S_{n-3}(y) \big)\,z^2 \nonumber
\end{eqnarray}
where $S_k(y)$'s are the Chebychev polynomials defined by $S_0(y)=1,~S_1(y)=y$ and $S_{k+1}(y)=yS_{k}(y)-S_{k-1}(y)$ for all integers $k$.
\label{universal}
\end{theorem}

\begin{proof}
The explicit formulas \eqref{42} and \eqref{43} follow from easy calculations of the trace polynomials using \eqref{41}.

Because $E$ and $F$ are conjugate (by $w^n$) and $w^nEa=Fw^na$ in $\pi_1(X)$, we have $P=Q_n=0$ for every representation $r: \pi_1(X) \to SL_2(\BC)$.

We will prove the converse: fix a solution $(x,y,z)$ of $P=Q_n=0$, we will find a representation $r: \pi_1(X) \to SL_2(\BC)$ such that $x=\tr(r(a)),~y=\tr(r(w))$ and $z=\tr(r(aw))$.

We consider the following 3 cases:

\smallskip

\underline{Case 1:} $~y^2 \not= 4$. Then there exist $s,u,v \in \BC$ such that $s+s^{-1}=y,~u+v=x,~su+s^{-1}v=z$. Since $S_k(y)=\frac{s^{k+1}-s^{-k-1}}{s-s^{-1}}$ for all integers $k$, we have
\begin{eqnarray*}
P &=& s^{-3}(s-1) P',\\
Q_n &=& s^{-3-n} \big( (s^{2n} u - s v)P'-(1 + s) (-1 + u v)Q'_n \big),
\end{eqnarray*}
where
\begin{eqnarray*}
P' &=& s^3 u-s^4 u-s^5 u+v+s v-s^2 v-s^2 u^2 v-s^3 u^2 v+s^4 u^2 v+s^5 u^2 v\\
&& - \, u v^2-s u v^2+s^2 u v^2+s^3 u v^2,\\
Q'_n &=& s^5+s^{2 n}-s^{2+2 n} u^2+s^{4+2 n} u^2+s^3 u v-s^5 u v-s^{2 n} u v+s^{2+2 n} u v+s v^2-s^3 v^2.
\end{eqnarray*}
Since $s \not= \pm 1$, $P=Q_n=0$ is equivalent to $P'=(-1+uv)Q'_n=0$. We consider the following 2 subcases:

\medskip

\underline{Subcase 1.1}: $Q'_n=0$. Choose $r(a)=\left(
\begin{array}{cc}
u & 1 \\
uv-1 & v \\
\end{array}
\right)$ and $r(w)=\left(
\begin{array}{cc}
s & 0 \\
0 & s^{-1} \\
\end{array}
\right)$. It is easy to check $x=\tr(r(a)),~y=\tr(r(w)),~z=\tr(r(aw))$ and the calculations in the following 2 lemmas.

\begin{lemma}
One has
$$r(E)=\left(
\begin{array}{cc}
s^{-2} H_{11} & -s^{-2} H_{12} \\
s^{-2}(-1 + u v) H_{21} & -s^{-2} H_{22} \\
\end{array}
\right), \quad r(F)=\left(
\begin{array}{cc}
-s^{-3} H_{22} & -s^{-1} H_{21} \\
s^{-3} (-1 + u v) H_{12} & s^{-1} H_{11} \\
\end{array}
\right)$$
where
\begin{eqnarray*}
H_{11} &=&  s^2 u - s^4 u + v - s^2 u^2 v + s^4 u^2 v - u v^2 + s^2 u v^2,\\
H_{12} &=&  1 - s^2 u^2 + s^4 u^2 - u v + s^2 u v,\\
H_{21} &=&  -s^4 - s^2 u v + s^4 u v - v^2 + s^2 v^2,\\
H_{22} &=&  -s^4 u + v - s^2 v - s^2 u^2 v + s^4 u^2 v - u v^2 + s^2 u v^2.
\end{eqnarray*}
\label{EF}
\end{lemma}

\begin{lemma}
One has $$r(w^nE-Fw^n)=\left(
\begin{array}{cc}
s^{-3 + n} P' & -s^{-2-n} Q'_n \\
-s^{-3-n} (-1+u v) Q'_n & -s^{-2-n} P' \\
\end{array}
\right).$$
\label{difference}
\end{lemma}

Since $P'=Q'_n=0$, Lemma \ref{difference} implies that $r(w^nE-Fw^n)=0$, i.e. $r(w^nE)=r(Fw^n)$.

\medskip

\underline{Subcase 1.2}: $-1+uv=0$ then $v=u^{-1}$. In this case the equation $P'=0$ becomes $s^2u^{-1}(s-u^2)=0$ i.e. $s=u^2$. Let $$r(a)=\left(
\begin{array}{cc}
u & 0 \\
0 & u^{-1} \\
\end{array}
\right), \quad r(w)=\left(
\begin{array}{cc}
u^2 & 0 \\
0 & u^{-2} \\
\end{array}
\right).$$
Then it is easy to check that $x=\tr(r(a)),~y=\tr(r(w)),~z=\tr(r(aw))$ and $r(Ew^n)=r(w^nF).$ (Note that $r(a)$ and $r(w)$ commute in this case).

\medskip

\underline{Case 2:} $~y=2$. Then $S_k(y)=k$ for all integers $k$. Hence
\begin{eqnarray*}
P &=& (x-z)(-1+x z-z^2),\\
Q_n &=& 4-(n-1)x^2+(3n-5)xz-(2n-3)z^2.
\end{eqnarray*}
It follows that $(x,z)=(-2,-2), (2,2)$ or ($x=z+z^{-1}$ and $1 - n + (1+ n) z^2 - z^4=0$).

If $x=z=2$ we choose $$r(a)=\left(
\begin{array}{cc}
1 & 0 \\
0 & 1 \\
\end{array}
\right), \quad r(w)=\left(
\begin{array}{cc}
1 & 0 \\
0 & 1 \\
\end{array}
\right).$$

If $x=z=-2$ we choose $$r(a)=\left(
\begin{array}{cc}
-1 & 0 \\
0 & -1 \\
\end{array}
\right), \quad r(w)=\left(
\begin{array}{cc}
1 & 0 \\
0 & 1 \\
\end{array}
\right).$$

If $x=z+z^{-1}$ and $1 - n + (1+ n) z^2 - z^4=0$ we choose
$$r(a)=\left(
\begin{array}{cc}
z & 0 \\
-z^{-1} & z^{-1} \\
\end{array}
\right), \quad r(w)=\left(
\begin{array}{cc}
1 & 1 \\
0 & 1 \\
\end{array}
\right).$$

\begin{lemma}
One has
$$r(w^nE-Fw^n)=\left(
\begin{array}{cc}
0 & z^{-1}(-1 + n - (1+n)z^2+ z^4) \\
0 & 0 \\
\end{array}
\right)$$
\end{lemma}

\begin{proof}
By direct calculations we have $r(E)=\left(
\begin{array}{cc}
z & -2z+z^3 \\
0 & z^{-1}\\
\end{array}
\right),~ r(F)=\left(
\begin{array}{cc}
z & z^{-1}-z\\
0 & z^{-1}\\
\end{array}
\right)$ and $r(w^n)=\left(
\begin{array}{cc}
1 & n \\
0 & 1\\
\end{array}
\right)$. The lemma follows.
\end{proof}
Hence $x=\tr(r(a)),~y=\tr(r(w)),~z=\tr(r(aw))$ and $r(w^nE)=r(Fw^n)$.

\medskip

\underline{Case 3:} $~y=-2$. Then $S_k(y)=(-1)^k k$ for all integers $k$. Hence
\begin{eqnarray*}
P &=& 3 x + z + x^2 z + 2 x z^2 + z^3,\\
Q_n &=& (-1)^{n}\big( xP-(x+z)Q''_n \big)/2,
\end{eqnarray*}
where $Q''_n=x + 2 n x + 2 z + x^2 z + x z^2.$ It follows that the system $P=Q_n=0$ is equivalent to $P=(x+z)Q''_n=0$. We consider the following 2 subcases:

\smallskip

\underline{Subcase 3.1:} $~x+z=0$. Then it is easy to see that $P=0$ is equivalent to $x=z=0$. In this case we choose $$r(a)=\left(
\begin{array}{cc}
i & 0 \\
0 & -i \\
\end{array}
\right), \quad r(w)=\left(
\begin{array}{cc}
-1 & 0 \\
0 & -1 \\
\end{array}
\right)$$
where $i$ is the imaginary number.

\underline{Subcase 3.2:} $~x+z \not= 0$. Then $Q''_n=0$. Choose $$r(a)=\left(
\begin{array}{cc}
x/2 & (1-x^2/4)/(x+z) \\
-x-z & x/2 \\
\end{array}
\right), \quad r(w)=\left(
\begin{array}{cc}
-1 & -1 \\
0 & -1 \\
\end{array}
\right).$$

\begin{lemma}
One has
$$r(w^nE-Fw^n)=(-1)^n\left(
\begin{array}{cc}
nP-Q''_n & Q''_n/2 \\
0 & -(n-1)P+Q''_n \\
\end{array}
\right).$$
\label{x+z}
\end{lemma}

\begin{proof}
By direct calculations, we have $r(w^n) = (-1)^n \left(
\begin{array}{cc}
1 & n \\
0 &  1\\
\end{array}
\right)$ and
\begin{eqnarray*}
r(E) &=& \left(
\begin{array}{cc}
-(x + 2 z + x^2 z + x z^2)/2 & -\frac{4 + 3 x^2 + 4 x z + x^3 z + x^2 z^2}{4(x+z)} \\
(x + z) (1 + x z + z^2) &  (3 x + 2 z + x^2 z + x z^2)/2\\
\end{array}
\right),
\\
r(F) &=& \left(
\begin{array}{cc}
(x + 2 z + x^2 z + x z^2)/2 & -\frac{4 + 5 x^2 + 10 x z + 3 x^3 z + 4 z^2 + 5 x^2 z^2 + 2 x z^3}{4(x+z)} \\
(x + z) (1 + x z + z^2) &  (-5 x - 4 z - 3 x^2 z - 5 x z^2 - 2 z^3)/2.
\end{array}
\right)
\end{eqnarray*}
The lemma follows.
\end{proof}

Hence $x=\tr(r(a)),~y=\tr(r(w)),~z=\tr(r(aw))$ and $r(w^nE)=r(Fw^n)$ in all cases. It follows that the character variety of the pretzel knot $K_{2n+1}$ is exactly equal to the algebraic set $\{P=Q_n=0\}$.
\end{proof}

\subsection{The universal character ring} In this subsection, we will prove the following
\begin{theorem}
\label{reduced} The universal character ring of $K_{2n+1}$ is reduced and is equal to the ring $\BC[x,y,z]/(P,Q_n)$.
\end{theorem}

\begin{proof}
Suppose we have shown that the ring $\BC[x,y,z]/(P,Q_n)$ is reduced, then it is exactly the character ring $\BC[\chi(X)]$ of $K_{2n+1}.$

Recall that $\pi_1(X)=\la a,w \mid w^nE=Fw^n\ra$, and $F_2=\la a,w \ra$ is the free group on two generators $a,w$. It is known that the universal character ring of $F_2$ is the ring $\BC[x,y,z]$ where $x=\tr(r(a)),~y=\tr(r(w))$ and $z=\tr(r(aw))$ as above. The quotient map $h: F_2 \to \pi_1(X)$ induces the epimorphism $h_{*}: \BC[x,y,z] \to \varepsilon(\CS(X))$. Since $P,Q_n$ come from traces, they are contained in $\ker h_*.$

Since $\BC[\chi(X))]$ is the quotient of $\varepsilon(\CS(X))$ by its nilradical, we have the quotient homomorphism $\phi: \varepsilon(\CS(X)) \to \BC[\chi(X))]=\BC[x,y,z]/(P,Q_n)$. Then
$$\phi \circ h_*: \BC[x,y,z] \to \varepsilon(\CS(X)) \to \BC[\chi(\pi)]=\BC[x,y,z]/(P,Q_n)$$
is a homomorphism. It follows that $\ker h_* \subseteq (P,Q_n).$ Hence we must have $\ker h_*=(P,Q_n),$ which implies $\varepsilon(\CS(X)) \cong \BC[x,y,z]/(P,Q_n) \equiv \BC[\chi(X)]$.

In the remaining part of this section we will show that  the ring $\BC[x,y,z]/(P,Q_n)$ is reduced, i.e. the ideal $I_n:=(P,Q_n)$ is radical. The proof of this fact will be divided into several steps.

\subsubsection{ $\BC[x,y,z]/I_n$ is free over $\BC[x]$.}

\begin{lemma}
\label{0}
For every $x_0 \not= 0, \pm 2$, the polynomial $P \mid_{x=x_0}$ is irreducible in $\BC[y,z]$.
\end{lemma}

\begin{proof}
Assume that $P \mid_{x=x_0}$ can be decomposed as
\begin{equation}
z^3-x_0yz^2+(y^2+x_0^2-3)z+x_0(1-y)=\big( z+f_1 \big) \big( z^2-(x_0y+f_1)z+f_2 \big),
\label{x_0}
\end{equation}
where $f_j \in \BC[y].$ Equation \eqref{x_0} implies that $f_2-f_1(x_0y+f_1)=y^2+x_0^2-3$ and $f_1f_2=x_0(1-y)$.

If $f_1$ is a constant then $f_2=x_0(1-y)/f_1$ has $y$-degree 1. Hence $f_2-f_1(x_0y+f_1)$ has $y$-degree 1 also. It follows that $f_2-f_1(x_0y+f_1) \not= y^2+x_0^2-3.$

If $f_2$ is a constant then $f_1=x_0(1-y)/f_2$. Hence
\begin{eqnarray*}
f_2-f_1(x_0y+f_1) &=& f_2 - \big( \frac{x_0}{f_2}-\frac{x_0}{f_2}y \big) \big( \frac{x_0}{f_2}-\frac{x_0}{f_2}y+x_0y \big)\\
&=& \frac{x_0^2}{f_2} \big( 1-\frac{1}{f_2} \big) y^2 - \frac{x_0^2}{f_2} \big( 1-\frac{2}{f_2} \big) y+ \big ( f_2 - \frac{x_0^2}{f_2^2} \big).
\end{eqnarray*}
Then since $f_2-f_1(x_0y+f_1)=y^2+x_0^2-3$, we have $\frac{x_0^2}{f_2} \big( 1-\frac{1}{f_2} \big)=1$, $\frac{x_0^2}{f_2} \big( 1-\frac{2}{f_2} \big)=0$, and $f_2 - \frac{x_0^2}{f_2^2}=x_0^2-3$. This implies $x_0=0$ or $x_0=\pm 2$.
\end{proof}

\begin{lemma}
For every $x_0$, the polynomials $P \mid_{x=x_0}$ and $Q_n \mid_{x=x_0}$ are co-prime in $\BC[y,z]$.
\label{gcd}
\end{lemma}

\begin{proof}
If $x_0 \not= 0, \pm 2$ then, by Lemma \ref{0}, $P \mid_{x=x_0}$ is irreducible in $\BC[y,z]$. Lemma \ref{gcd} then follows since $P \mid_{x=x_0}$ and $Q_n \mid_{x=x_0}$ have $z$-degrees 3 and 2 respectively.

At $x_0=0$, we have $P=z(-3+y^2+z^2)$ and $Q_n=a_n+b_nz^2$ where
\begin{eqnarray*}
a_n &=& S_{n-2}(y)+S_{n-3}(y)-S_{n-4}(y)-S_{n-5}(y), \\
b_n &=& -S_{n-2}(y)-S_{n-3}(y).
\end{eqnarray*}
In this case, it suffices to show that $Q_n \mid_{z^2=3-y^2} = a_n+b_n(3-y^2) \not= 0$. This is true by Lemma \ref{T'} below.

At $x_0=2$, we have $P=\big( z+1-y \big) \big( z^2-(1+y)z+2 \big)$ and $Q_n=a_n'+b_n'z+c_n'z^2$ where $a'_n, b'_n, c'_n \in \BC[y]$. When $z=y-1$, we have $Q_0=1$ and $Q_1=y-1$ and $Q_{n+1}=yQ_n-Q_{n-1}$ for all integers $n$. It follows that $Q_n\mid_{z=y-1}=S_{n}(y)-S_{n-1}(y)$ is a polynomial of $y$-degree $n$ if $n \ge 0$ and $-(n+1)$ if $n \le -1$, with leading coefficient 1. Hence $Q_n\mid_{z=y-1}$ is not identically 0. It remains to show that $Q_n=a_n'+b_n'z+c_n'z^2 \not= c'_n (z^2-(1+y)z+2)$. It suffices to show that $b'_n \mid_{y=-1} \not =0$. Indeed, when $x_0=2$ and $y=-1$ we have $b'_n=2 \big( S_{n-1}(-1)+S_{n-3}(-1)+ S_{n-4}(-1)\big).$ It is easy to check that $S_k(-1)=1$ if $k=0  \pmod{3}$, $S_k(-1)=-1$ if $k=1  \pmod{3}$ and $S_k(-1)=0$ otherwise. Hence $b'_n=2 \big( S_{n-1}(-1)+S_{n-3}(-1)+ S_{n-4}(-1)\big) \not= 0$.

The case $x_0=-2$ is similar.
\end{proof}

\begin{proposition}
$\BC[x,y,z]/I_n$ is a torsion-free $\BC[x]$-module.
\label{torsion-free}
\end{proposition}

\begin{proof}
Suppose $S \in \BC[x,y,z]$ and $(x-x_0)S \in I_n$ for some $x_0 \in \BC$. We will show that $S \in I_n.$ Indeed, we have $(x-x_0)S=fP-gQ_n$ for some $f,g \in \BC[x,y,z]$. Hence $(fP)\mid_{x=x_0}=(gQ_n)\mid_{x=x_0}$ which implies that $f\mid_{x=x_0}$ is divisible by $Q_n \mid_{x=x_0}$, since $P \mid_{x=x_0}$ and $Q_n \mid_{x=x_0}$ are co-prime in the UFD $\BC[y,z]$ by Lemma \ref{gcd}. Hence $f_{x=x_0}=h \, Q_n\mid_{x=x_0}$ for some $h \in \BC[y,z]$. From this, we may write $f=h\,Q_n+(x-x_0)Q$ for some $Q \in \BC[x,y,z]$. Then we have
$$(x-x_0)S=fP-gQ_n=\big( hQ_n+(x-x_0)Q \big)P - gQ_n=(x-x_0)QP+(hP-g)Q_n$$
which implies that $hP-g$ is divisible by $x-x_0$ and $S=QP+\frac{hP-g}{x-x_0}Q_n \in I_n$.
\end{proof}

\begin{proposition}
$\BC[x,y,z]/I_n$ is a finitely generated $\BC[x]$-module.
\label{fg}
\end{proposition}

\begin{proof}
We want to show that $y$ and $z$, considered as elements of $\BC[x,y,z]/I_n$, are integral over $\BC[x]$. Indeed, the resultant of $P$ and $Q_n$ with respect to $z$ is
$$\mathfrak r = \left| \begin{array}{ccccc}
P_0 & P_1 & P_2 & P_3 & 0 \\
0 & P_0 & P_1 & P_2 & P_3 \\
Q_{n,0} & Q_{n,1} & Q_{n,2} & 0 & 0\\
0 & Q_{n,0} & Q_{n,1} & Q_{n,2} & 0\\
0 & 0 & Q_{n,0} & Q_{n,1} & Q_{n,2}
\end{array} \right|$$
where $P_0 =x - x y,~ P_1 =-3 + x^2 + y^2,~P_2 = -xy,~P_3=1$ and
\begin{eqnarray*}
Q_{n,0} &=& S_{n-2}(y)+S_{n-3}(y)-S_{n-4}(y)-S_{n-5}(y)-S_{n-2}(y)x^2,\\
Q_{n,1} &=& ( S_{n-1}(y)+S_{n-3}(y)+S_{n-4}(y))x,\\
Q_{n,2} &=& -(S_{n-2}(y)+S_{n-3}(y)).
\end{eqnarray*}
Write $y=s+s^{-1}$ then $S_k(y)=\frac{s^{k+1}-s^{-k-1}}{s-s^{-1}}$ for all integers $k$. By a direct calculation
\begin{eqnarray*}
\mathfrak r &=& \frac{s+s^{-1}+2-x^2}{(s-s^{-1})^2(s+s^{-1}+2)}\big(s^{3n}+s^{-3n}+3s^{3n-1}+3s^{1-3n}+3s^{3n-2}+3s^{2-3n}\\
&& + \, s^{3n-3}+s^{3-3n}+s^{n+5}+s^{-n-5}+3s^{n+4}+3s^{-n-4}+3s^{n+3}+3s^{-n-3}+s^{n+2}+s^{-n-2}\\
&& - \, 2s^{n-1}-2s^{1-n}-6s^{n-2}-6s^{2-n}-6s^{n-3}-6s^{3-n}-2s^{n-4}-2s^{4-n}\\
&& + \, x^2(-s^{3n-1}-s^{1-3n}-s^{3n-2}-s^{2-3n}-2s^{n+3}-2s^{-n-3}-3s^{n+2}-3s^{-n-2}\\
&& - \, s^{n+1}-s^{-n-1}-5s^{n}-5s^{-n}-2s^{n-1}-2s^{1-n}+8s^{n-2}+8s^{2-n}+6s^{n-3}+6s^{3-n}\\
&& + \, s^{n-4}+s^{4-n})+x^4(s^{n+1}+s^{-n-1}+2s^n+2s^{-n}-2s^{n-2}-2s^{2-n}-s^{n-3}-s^{3-n})\big).
\end{eqnarray*}
Let $T_k(y)=s^k+s^{-k}$ for all intergers $k$. Then we have
\begin{eqnarray*}
\mathfrak r &=& \frac{y+2-x^2}{(y^2-4)(y+2)}\big( T_{3n}(y)+3T_{3n-1}(y)+3T_{3n-2}(y)+T_{3n-3}(y)+T_{n+5}(y)\\
&& + \, 3T_{n+4}(y)+ 3T_{n+3}(y)+T_{n+2}(y)-2T_{n-1}(y)-6T_{n-2}(y)-6T_{n-3}(y)-2T_{n-4}(y)\\
&& + \, x^2(-T_{3n-1}(y)-T_{3n-2}(y)-2T_{n+3}(y)-3T_{n+2}(y)-T_{n+1}(y)-5T_{n}(y)-2T_{n-1}(y)\\
&& + \, 8T_{n-2}(y)+6T_{n-3}(y)+T_{n-4}(y))+x^4(T_{n+1}(y)+2T_n(y)-2T_{n-2}(y)-T_{n-3}(y))\big)
\end{eqnarray*}

Note that $T_k(y)$ has $y$-degree $|k|$ with leading coefficient 1. If $n \ge 4$ then it is easy to see that $\mathfrak r$ has $y$-degree $3n-2$; moreover the coefficient of $y^{3n-2}$ is 1. Similarly, if $n \le -5$ then $\mathfrak r$ has $y$-degree $1-3n$; moreover the coefficient of $y^{1-3n}$ is 1. If $-4 \le n \le 3$ then by direct calculations, one can check that the coefficient of the highest power of $y$ in $\mathfrak r$ is 1. Hence the coefficient of the highest power of $y$ in $\mathfrak r$ is 1 for all integers $n$. It follows that $y$, considered as an element of $\BC[x,y,z]/I_n$, is integral over $\BC[x]$. 

Since $z$, considered as an element of $\BC[x,y,z]/I_n$, satisfies the equation $P=x-xy+(-3+x^2+y^2)z^2-xyz^2+z^3=0$ with 1 being the coefficient of the highest power of $z$, it is also integral over $\BC[x]$. Therefore $\BC[x,y,z]/I_n$ is a finitely generated $\BC[x]$-module.
\end{proof}

Since $\BC[x]$ is a PID, Propositions \ref{torsion-free} and \ref{fg} imply that

\begin{proposition}
$\BC[x,y,z]/I_n$ is a free $\BC[x]$-module.
\label{free}
\end{proposition}

\subsubsection{Reduction to a special case} For a $\BC[x]$-module $J$, let $J\mid_{x=x_0}:= J \otimes_{\BC[x]} \BC$, where $\BC$ is considered as an $\BC[x]$-module by reducing $x=x_0.$

\begin{proposition}
$I_n$ is radical if $I_n\mid_{x=x_0}$ is radical for some $x_0 \in \BC$.
\label{x0}
\end{proposition}

\begin{proof}
Let $R=\BC[x,y,z]$. Consider the exact sequence of $\BC[x]$-modules $$0 \to \sqrt{I_n}/I_n \to R/\sqrt{I_n} \to R/I_n \to 0.$$ By Proposition \ref{free}, $R/I_n$ is free, hence the sequence splits and $\sqrt{I_n}/I_n$ is projective. Since $\BC[x]$ is a PID, $\sqrt{I_n}/I_n$ is free. Let $k$ be the rank of the $\BC[x]$-module $\sqrt{I_n}/I_n$ then the rank of the $\BC$-module $(\sqrt{I_n}/I_n)\mid_{x=x_0}$ is always $k$ for every $x_0 \in \BC$. Hence if $I_n\mid_{x=x_0}$ is radical for some $x_0 \in \BC$ then $k=0$ which implies that $\sqrt{I_n}=I_n$.
\end{proof}

\subsubsection{ $I_n\big |_{x=0}$ is radical} By Lemma \ref{gcd}, $P\big|_{x=0}$ and $Q_n\big|_{x=0}$ are co-prime in $\BC[y,z]$. This means $I_n\big |_{x=0}$ is a zero-dimensional ideal of $\BC[y,z]$.
By Seidenberg's Lemma (see \cite[Proposition 3.7.15]{KR}), if
 there exist two non-zero free-square polynomials in $I_n\big |_{x=0} \, \cap \, \BC[y]$ and $I_n\big |_{x=0} \, \cap \, \BC[z]$ respectively, then $I_n\big |_{x=0}$ is radical.

From now on we fix $x=0$. Then $P=z(-3+y^2+z^2)$ and $Q_n=a_n+b_nz^2$ where
\begin{eqnarray*}
a_n &=& S_{n-2}(y)+S_{n-3}(y)-S_{n-4}(y)-S_{n-5}(y), \\
b_n &=& -S_{n-2}(y)-S_{n-3}(y).
\end{eqnarray*}
Let $U_n=a_n+b_n(3-y^2).$ Then $U_0=1, U_1=y+1$ and $U_{n+1}=yU_{n}-U_{n-1}.$ Hence $$U_n=S_{n}(y)+S_{n-1}(y).$$

We first consider the case $n \ge 3.$ Then $U_n$ and $a_n$ have $y$-degrees $n$ and $n-2$ respectively; moreover their leading coefficients are equal to 1.

\begin{lemma}
 One has
$$U_n=\prod_{j=1}^{n}(y-2\cos \frac{j2\pi}{2n+1}).$$
\label{T'}
\end{lemma}

\begin{proof}
It is easy to see that $U_n$ is a polynomial of degree $n$ in $y$. Note that if $y=s+s^{-1} \not= \pm 2$ then $S_k(y)=\frac{s^{k+1}-s^{-k-1}}{s-s^{-1}}$. We now take $y=e^{i\frac{j2\pi}{2n+1}}+e^{-i\frac{j2\pi}{2n+1}}=2\cos \frac{j2\pi}{2n+1}$ where $1 \le j \le n$. Then
$$S_{n}(y)=\frac{\sin ((n+1)\frac{j2\pi}{2n+1})}{\sin (\frac{j2\pi}{2n+1})}=-\frac{\sin (n\frac{j2\pi}{2n+1})}{\sin (\frac{j2\pi}{2n+1})}=-S_{n-1}(y).$$
The lemma follows.
\end{proof}

\begin{lemma}
One has
$$a_n=\prod_{k=0}^{n-3}(y-2\cos \frac{(2k+1)\pi}{2n-5}).$$
\end{lemma}

\begin{proof}
The proof is similar to that of the previous lemma.
\end{proof}

Note that $$b_n^2\, z \, P=b_n\, z^2 \left( (-3+y^2)b_n+b_nz^2 \right)=(Q_n-a_n)(Q_n-U_n).$$ Hence $a_n \, U_n=a_n \, Q_n-Q_n^2+Q_n \, U_n+b_n^2\,z\,P$ is contained in $I_n\mid_{x=0}$.
But $a_n \, U_n$ is a polynomial in $y$, hence it is actually contained in $I_n\mid_{x=0} \,\cap \,\BC[y].$ It is easy to see that $a_n \, U_n$ is square-free, i.e. does not have repeated factors.

Let
$$V_n=z\prod_{j=1}^n (-3+4\cos^2 \frac{j2\pi}{2n+1}+z^2)\prod_{k=0}^{n-3}(-3+4\cos^2 \frac{(2k+1)\pi}{2n-5}+z^2).$$
Then it is easy to show that $V_n \in \BC[z]$ is square-free. Moreover, since
\begin{eqnarray*}
V_n &=& z\prod_{j=1}^n (-3+y^2+z^2+(4\cos^2 \frac{j2\pi}{2n+1}-y^2))\\
&& \times \, \prod_{k=0}^{n-3}(-3+y^2+z^2+(4\cos^2 \frac{(2k+1)\pi}{2n-5}-y^2))\\
    &\equiv& z\prod_{j=1}^n (4\cos^2 \frac{j2\pi}{2n+1}-y^2)\prod_{k=0}^{n-3}(4\cos^2 \frac{(2k+1)\pi}{2n-5}-y^2) \quad (\text{mod}~P),\\
    &\equiv& 0 \quad (\text{mod}~(P,a_nU_n))
\end{eqnarray*}
it is contained in $I_n\mid_{x=0}.$ Hence $V_n$ is in $I_n\mid_{x=0} \, \cap \, \BC[z]$ and is square-free.

Since both $a_n \, U_n \in I_n\mid_{x=0} \cap \, \BC[y]$ and $V_n \in I_n\mid_{x=0} \cap \, \BC[z]$ are square-free, $I_n\mid_{x=0}$ is a radical ideal by Seidenberg's Lemma. Hence by Proposition \ref{x0}, $I_n$ is also radical. It follows that $R/I_n$ is reduced. Hence the ring $\BC[x,y,z]/(P,Q_n)$ is reduced and is equal to the universal character ring of $K_{2n+1}.$ This proves Theorem \ref{reduced} for the case $n \ge 3$. The case $n \le -1$ is similar (in this case $U_n$ and $a_n$ have $y$-degrees $3-n$ and $-(n+1)$ respectively; moreover their leading coefficients are equal to 1 and $-1$ respectively). If $0 \le n \le 2$ then by direct calculations one can check that $I_n\mid_{x=0}$ is reduced. This completes the proof of Theorem \ref{reduced} for all integers $n$.
\end{proof}

\section{The skein module of the $(-2,3,2n+1)$-pretzel knot}

Let $K$ be the $(-2,3,2n+1)$-pretzel knot and $\CS$ be the skein module of $S^3 \setminus K$. Recall from Section \ref{AJ} that $\CR=\BC[t^{\pm 1}]$, $\bD$ is the localization of $D=\CR[M^{\pm 1}]$ at the ideal $(1+t)$, and $\bS=\CS \otimes_{D^{\sigma}} \bD$. The goal of this section is to prove the following.

\begin{proposition}
The localized skein module $\bS$ is a finitely generated $\bD$-module.
\label{localized}
\end{proposition}

For $n=0,1$ or $2$, $K$ is a torus knot and hence its skein module $\CS$ has been understood in \cite{Mar}. Hence we consider $n \ge 3$ or $n \le -1$ only.

\subsection{Knot complement} Consider the genus 2 handlebody $H_2$, which is presented as the cylinder $D^2 \times [0,1]$ minus two vertical tubes as in Figure \ref{H2}.
\begin{figure}[htpb]
$$\psdraw{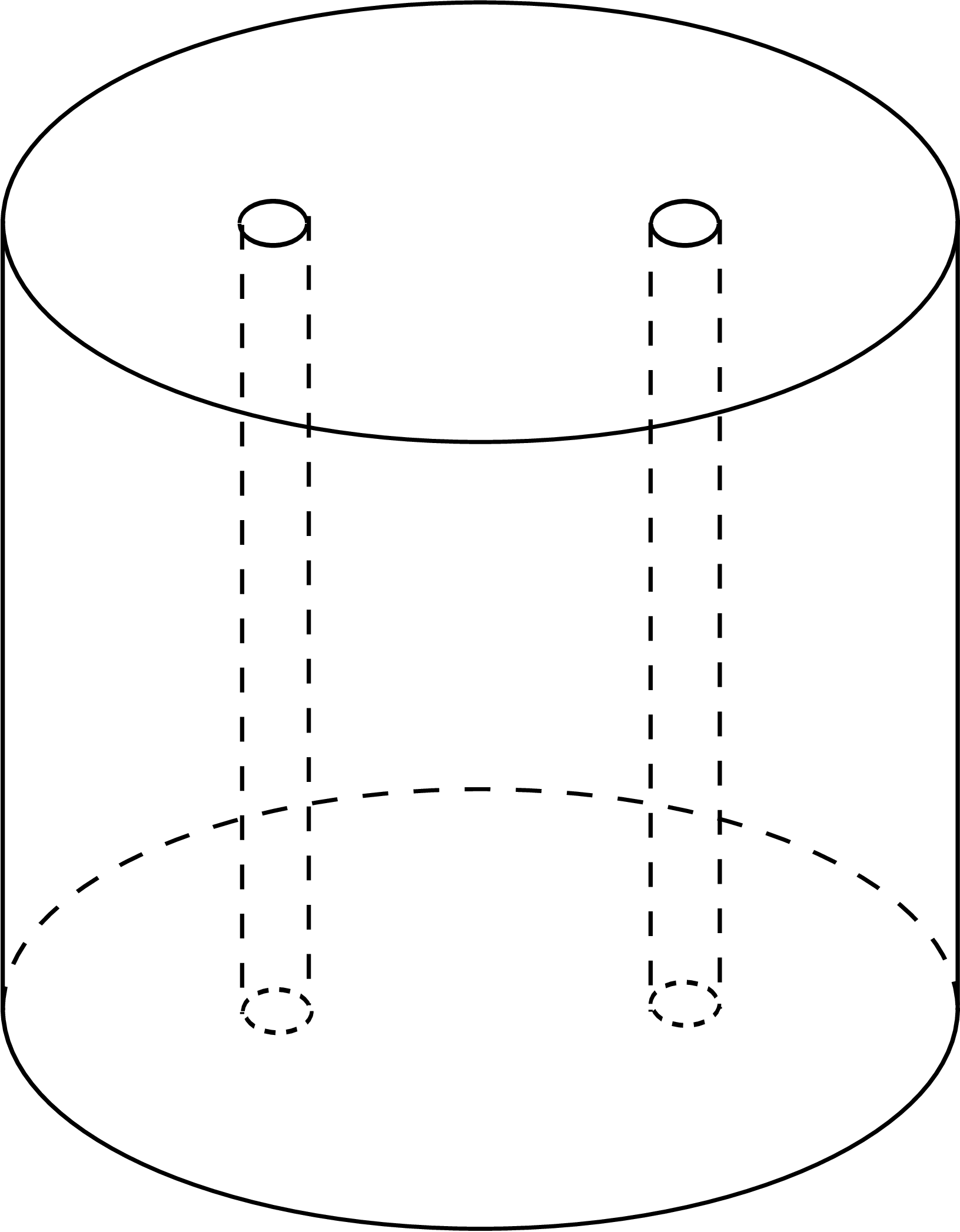}{1.5in} $$
\caption{The genus 2 handlebody $H_2$}
\label{H2}
\end{figure}

Let $x$ (resp. $y$) be a small loop on $\partial H_2$ circling the top of the left (resp. right) hand side tube of $\partial H_2$, and let $z$ is a small loop on $\partial H_2$ circling the top of the left hand side tube and the top of the right hand side tube of $\partial H_2$ as in Figure \ref{xyz}.

\begin{figure}[htpb]
$$\psdraw{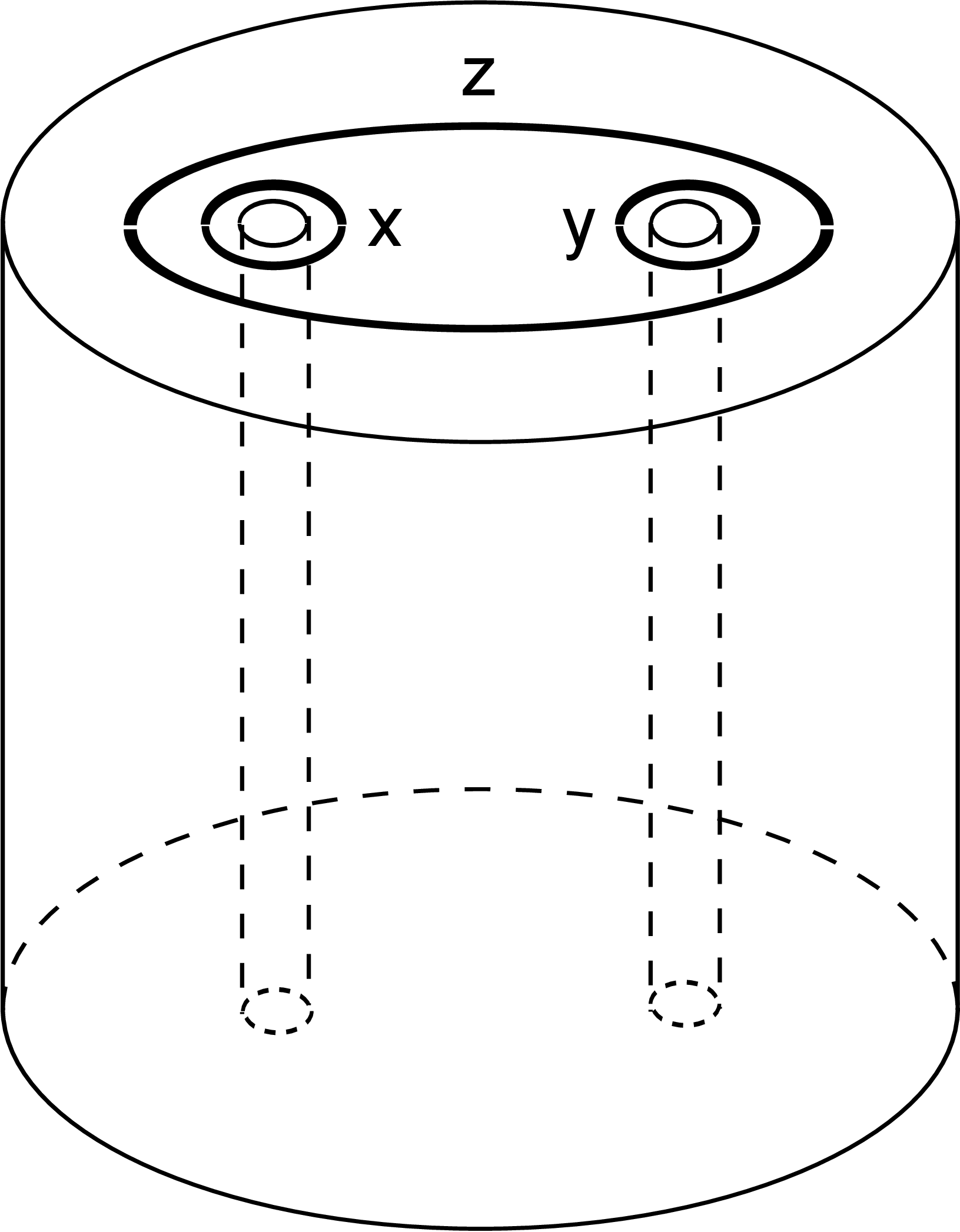}{1.5in} $$
\caption{The loops $x,y$ and $z$ on $\partial H_2$}
\label{xyz}
\end{figure}

For $n \ge 3$ (resp. $n \le -1$) we consider the closed curve $C$ on $\partial H_2$ as in Figure \ref{curve} (resp. Figure \ref{curve'}), where the part of $C$ along the right hand side tube of $\partial H_2$ is drawn in the right hand side part of Figure \ref{curve} (resp. Figure \ref{curve'}). We also choose the point $pt$ on $C$ as in Figure \ref{curve} (or Figure \ref{curve'}) and consider it as the based point in $H_2$.

\begin{figure}[htpb]
$$\psdraw{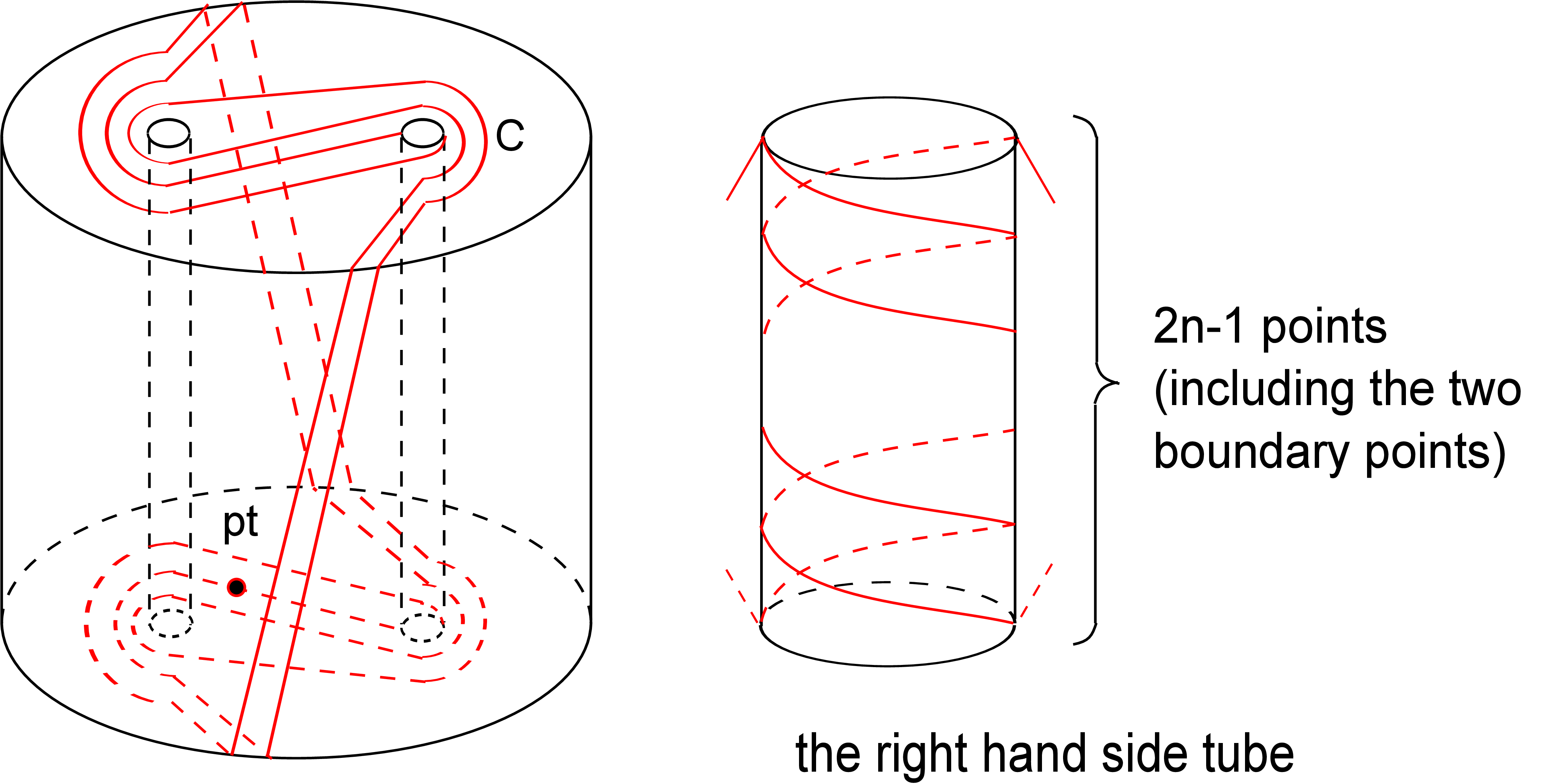}{4in} $$
\caption{The closed curve $C$ on $\partial H_2$ and the based point $pt$, for $n \ge 3$.}
\label{curve}
\end{figure}

\begin{figure}[htpb]
$$\psdraw{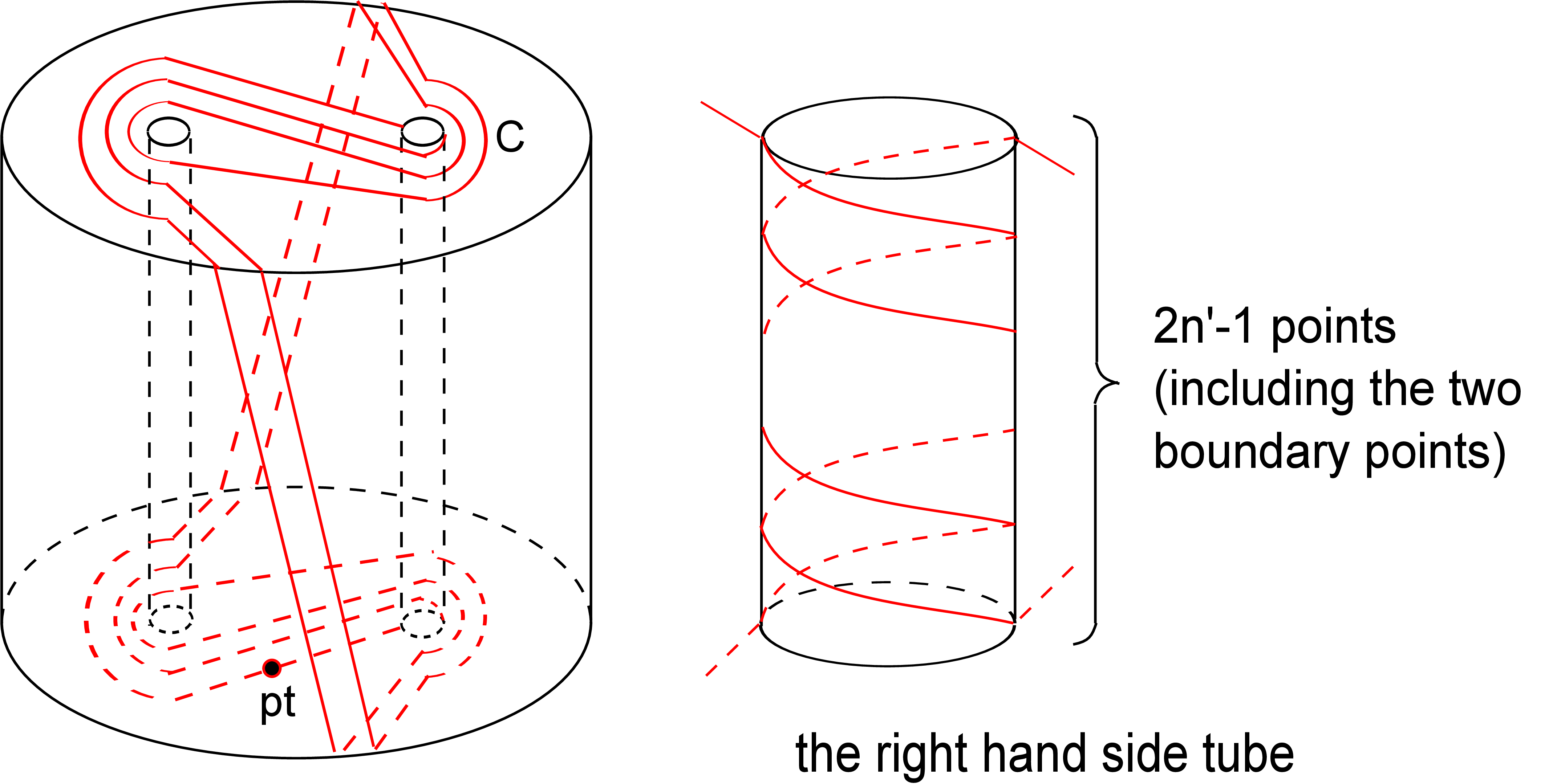}{4in} $$
\caption{The closed curve $C$ on $\partial H_2$ and the based point $pt$, for $n \le -1$. Here $n'=1-n$.}
\label{curve'}
\end{figure}

Let $X$ be the 3-manifold obtained by attaching a 2-handle along $C$ to $H_2$.

\begin{lemma} There is a homeomorphism between
$X$  and the complement of the $(-2,3,2n+1)$-pretzel knot $K$ in $S^3$ under which $x$ is mapped to a meridian of $K$.
\end{lemma}

\begin{proof}
Recall from Section \ref{pretzel} that $$\pi_1(S^3 \setminus K) = \la a,w \mid w^nawa^{-1}w^{-1}a^{-1}=a^{-1}w^{-1}awaw^{-1}w^n\ra$$
where $a$ is a meridian of $S^3 \setminus K$.

We have $\pi_1(H_2)=\la x,y \mid \ra$. In $\pi_1(H_2)$, one has $$C=y^{1-n}x^{-1}y^{-1}x^{-1}yxy^nxyx^{-1}y^{-1}x^{-1} \text{~if~} n \ge 3$$ and $$C=y^{1-n'}xyx^{-1}y^{-1}x^{-1}y^{n'}x^{-1}y^{-1}x^{-1}yx \text{~if~} n \le -1.$$
Here $n'=1-n.$ Since $X$ is obtained by attaching a 2-handle along $C$ to $H_2$,
$$\pi_1(X)=\la x,y \mid y^{1-n}x^{-1}y^{-1}x^{-1}yxy^nxyx^{-1}y^{-1}x^{-1} =1\ra.$$
which is isomorphic to $\pi_1(S^3 \setminus K)$ via the isomorphism sending $x,y$ to $a, w$ respectively.

Let $H_2'$ be the 3-manifold obtained by attaching a 2-handle along $x$ to $H_2$. Then $H_2'$ is just a solid torus whose core is homotopic in $H_2'$ to $y$. In $\pi_1(H_2')$, one has $x=1$ and $C=y$. It implies that the 3-manifold obtained by attaching a 2-handle along $C$ to $H_2'$ is a 3-dimensional ball $D^3$.

Note that this 3-manifold is also obtained by attaching a 2-handle along $x$ to $X$, since $x$ and $C$ are disjoint. Hence the 3-manifold obtained by attaching a 2-handle along $x$ to $X$ is $D^3$. Let $D'^3=\overline{S^3-D^3}$. Then $D'^3$ is another 3-dimensional ball in $S^3$, and $\partial D'^3=\partial D^3$ is a 2-sphere. The complement of $X$ in $S^3$ is the 3-manifold obtained by attaching a 2-handle along $x$ to $D'^3$, and hence is a solid torus. It follows that $X$ is the complement of a knot $K'$ in $S^3$.

Since $S^3 \setminus K'$ and $S^3 \setminus K$ have isomorphic fundamental groups and $K$ is a prime knot, a theorem of Gordon and Lueck \cite{GoLu} implies that $K'$ is equivalent to $K$.
\end{proof}

From now on, we identify $X$ with the knot complement $S^3 \setminus K$. Note that $x$ is a meridian of $X$.

The proofs of Proposition \ref{localized} in the cases $n \ge 3$ and $n \le -1$ are similar. Hence in the remaining part of this section, without of loss of generality we assume that $n \ge 3$.

\subsection{Skein module} We now describe the skein module $\CS$ as a quotient of $\CR[x,y,z]$ by a submodule.

\label{skein module}
{\rm  A type 1 tangle} in a 3-manifold $Y$ (with boundary) is the disjoint union of a framed link and a
framed arc in $Y$ such that  the parts of the arc near the two end
points are on the boundary $\partial Y$, and the framing on these
parts are given by vectors normal to $\partial Y$. Type 1 tangles
are considered up to isotopy relative the endpoints.

Recall that $X$ is obtained from $H_2$ by attaching a 2-handle
along the closed curve $C$. Note that $\CS(H_2)$ is isomorphic to the commutative algebra
$\CR[x,y,z]$ where $x,y$ and $z$ are as defined as above, see \cite{Pr}. The embedding of $H_2$ into $X$ gives rise to
a linear map from $\CS(H_2)\equiv \CR[x,y,z]$ to $\CS(X)$. It is
known that the map is surjective, and its kernel $N$, see
\cite{Pr,BL}, can be described through slides as
follows.

Suppose $\fa$ is a type 1 tangle in $H_2$ whose 2 endpoints are on $C$ such
that outside a small neighborhood of the 2 endpoints $\fa$ is in the
interior of $H_2$ and in a small neighborhood of the endpoints $\fa$
is on the boundary $\partial H_2$. The two end points of $\fa$
divide $C$ into 2 arcs $C_1$ and $C_2$. Let $sl(\fa)$ be $\fa \cdot C_1-\fa \cdot C_2$, considered as an element
of the skein module $\CS(H_2)$. Here $\fa \cdot C_i$ is the framed link
obtained by combining  $\fa$ and $C_i$.

It is clear that as framed links in $X$, $\fa \cdot C_1$ is isotopic to
$\fa \cdot C_2$, since one is obtained from the other by sliding over
the 2-handle attached to the curve $C$. Hence we always have
$sl(\fa)\in N$. It was known that $N$ is spanned by all possible
$sl(\fa)$, where $\fa$ is chosen among all type 1 tangles in $H_2$ with
endpoints on $C$.

There is a natural bilinear map
$\CS(\partial H_2) \otimes \CS(H_2) \to \CS(H_2)$, where $\ell
\otimes \ell' \to \ell \star \ell'$, which  is the disjoint union
of $\ell$ and $\ell'$. Hence $N$ contains all $sl(\fa) \star \CS(H_2)$, where $\fa$ is any type 1 tangles in $\partial H_2 \times [0,1]$ with endpoints on $C$.

Consider the points $u,v,u'$ and $v'$ on the curve $C$ as in Figure \ref{PQ}. Let $\tilde{P}=sl({\overline{uv}})$ and $\tilde{Q}=sl({\overline{u'v'}})$ be skein elements in $\CS(\partial H_2)$, where $\overline{uv}$ (resp. $\overline{u'v'}$) is the straight segment on $\partial H_2$ connecting $u,v$ (resp. $u',v'$) whose interior is slightly pushed inside
the interior of $H_2$ (to avoid intersections with other arcs on $\partial H_2$) and whose framing is given by vector normal to $\partial H_2$.

\begin{figure}[htpb]
$$\psdraw{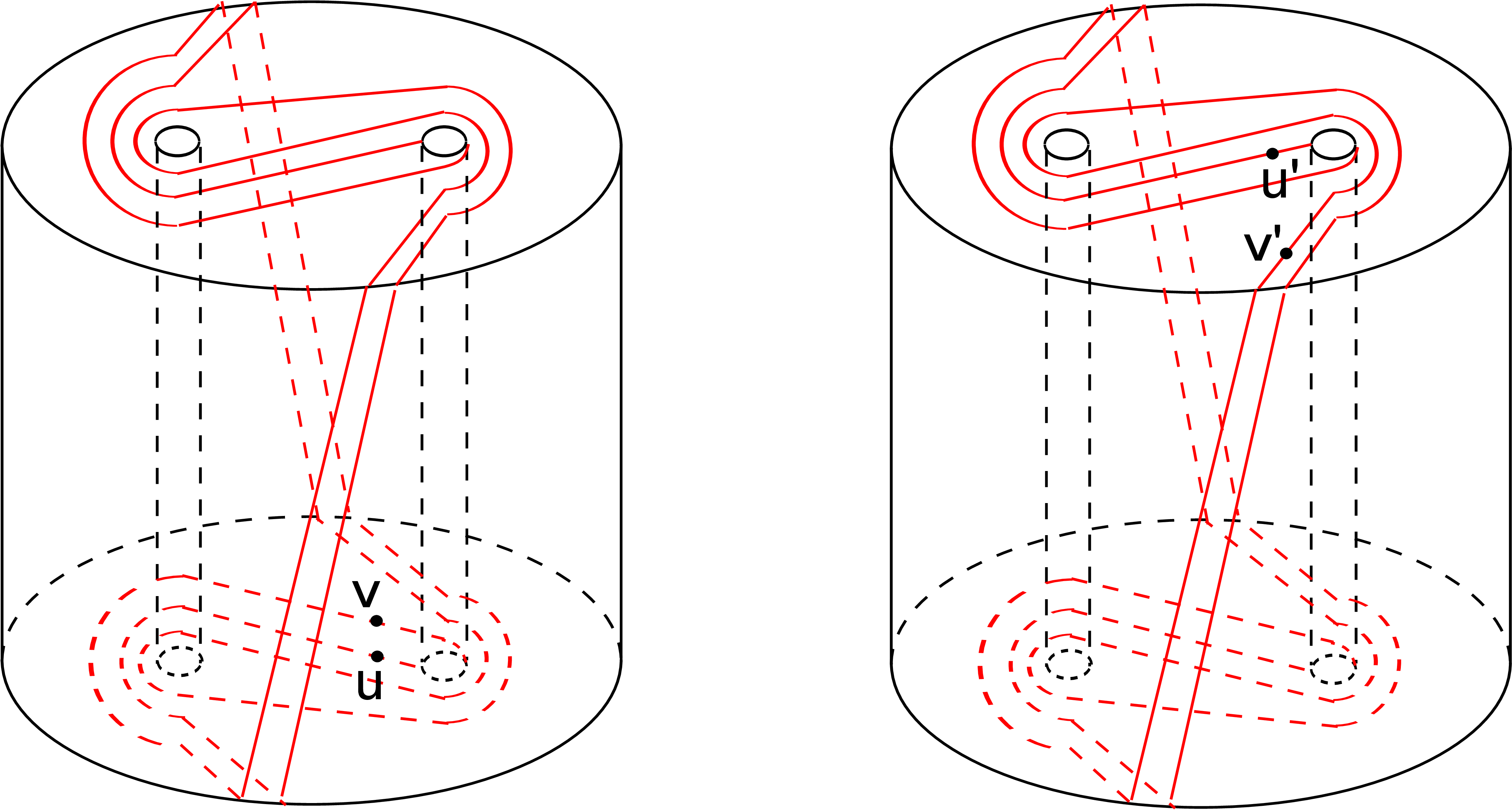}{4in} $$
\caption{The points $u,v,u'$ and $v'$ on $C$}
\label{PQ}
\end{figure}

From the discussion above, one has

\begin{proposition}
\label{N}
The skein module of the complement of the $(-2,3, 2n+1)$-pretzel knot is $\CR[x,y,z]/N$, where $N$ is an $\CR[x]$-submodule of $\CR[x,y,z]$ containing all $\tP* y^k z^l$ and $\tQ * y^k z^l$ with $k, l \ge0 $.
\end{proposition}

By isotopies in $\partial H_2$ one can check that $\tP= \tP_1- \tP_2$ in $\CS(\partial H_2)$, where $\tP_1$ and $\tP_2$ are curves on $\partial H_2$ depicted in Figure \ref{P1P2}.

\begin{figure}[htpb]
$$\psdraw{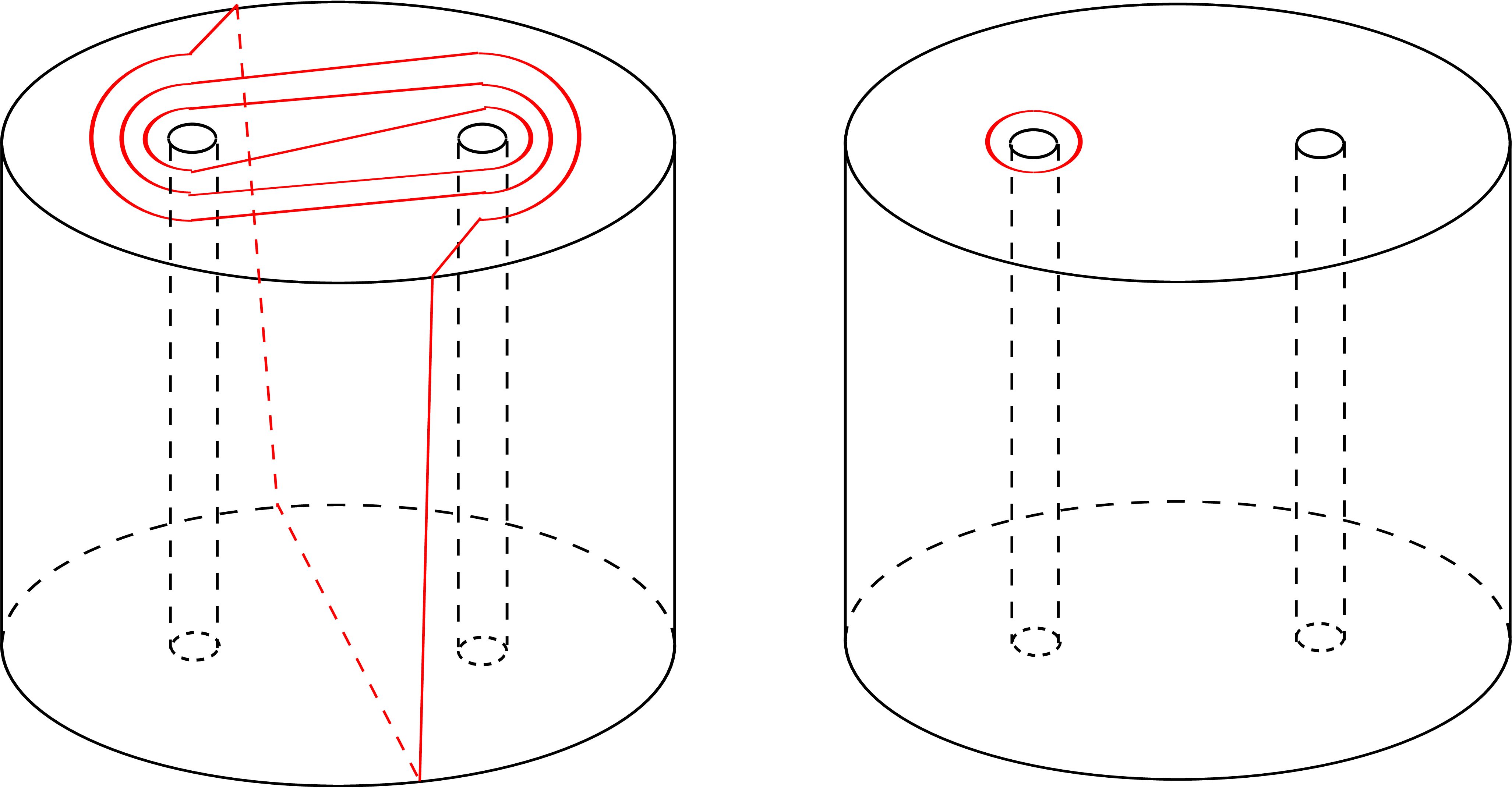}{3.8in}$$
\caption{The curve $\tP_1$ on the left and the curve $\tP_2$ on the right}
\label{P1P2}
\end{figure}

Similarly $\tQ= \tQ_1- \tQ_2$ in $\CS(\partial H_2)$, where $\tQ_1$ and $\tQ_2$ are curves on $\partial H_2$ depicted in Figure \ref{Q1Q2}. Note that Figure \ref{k} explains the notation used in Figure \ref{Q1Q2}.

\begin{figure}[htpb]
$$\psdraw{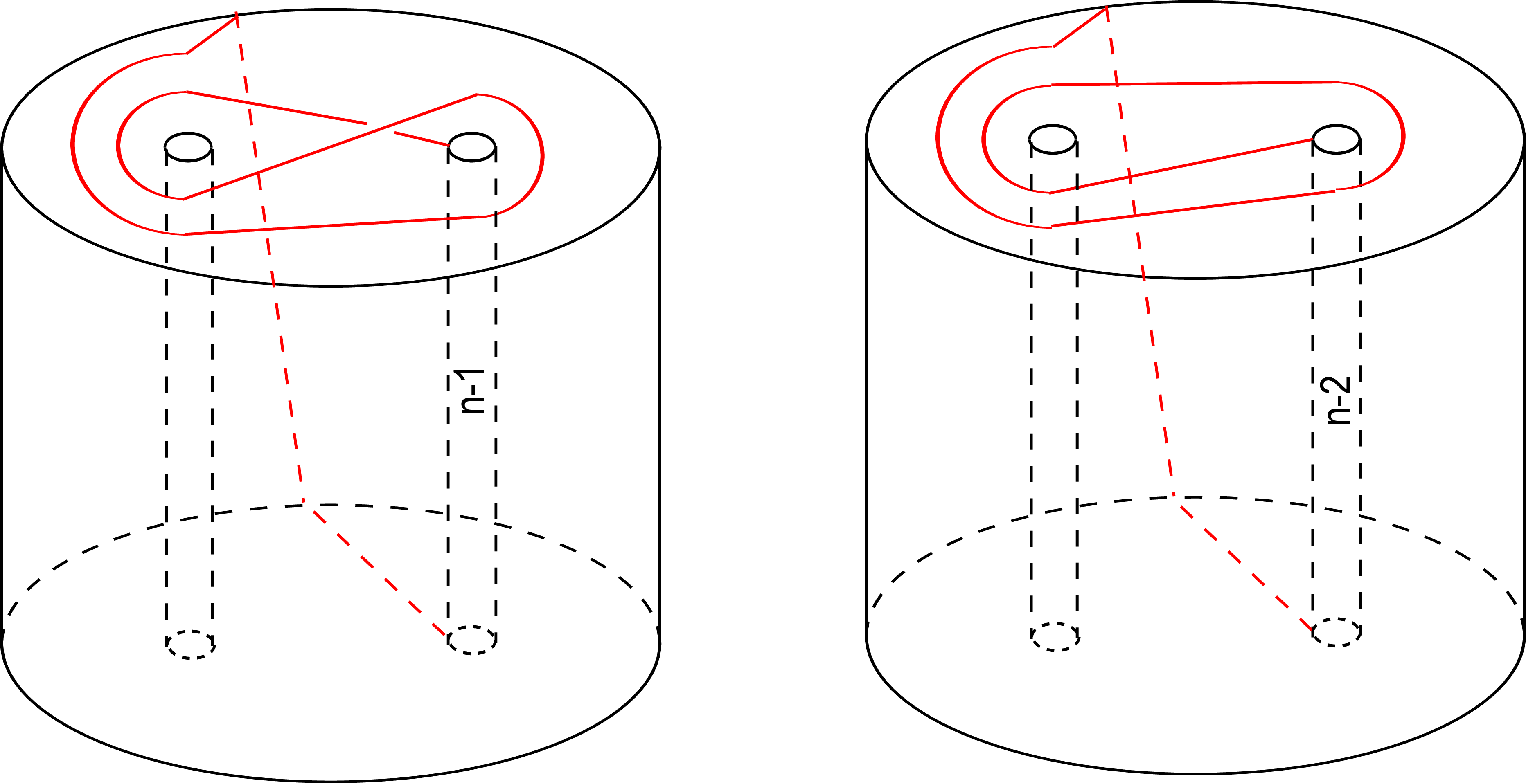}{4in} $$
\caption{The curve $\tQ_1$ on the left and the curve $\tQ_2$ on the right}
\label{Q1Q2}
\end{figure}
\begin{figure}[htpb]
$$\psdraw{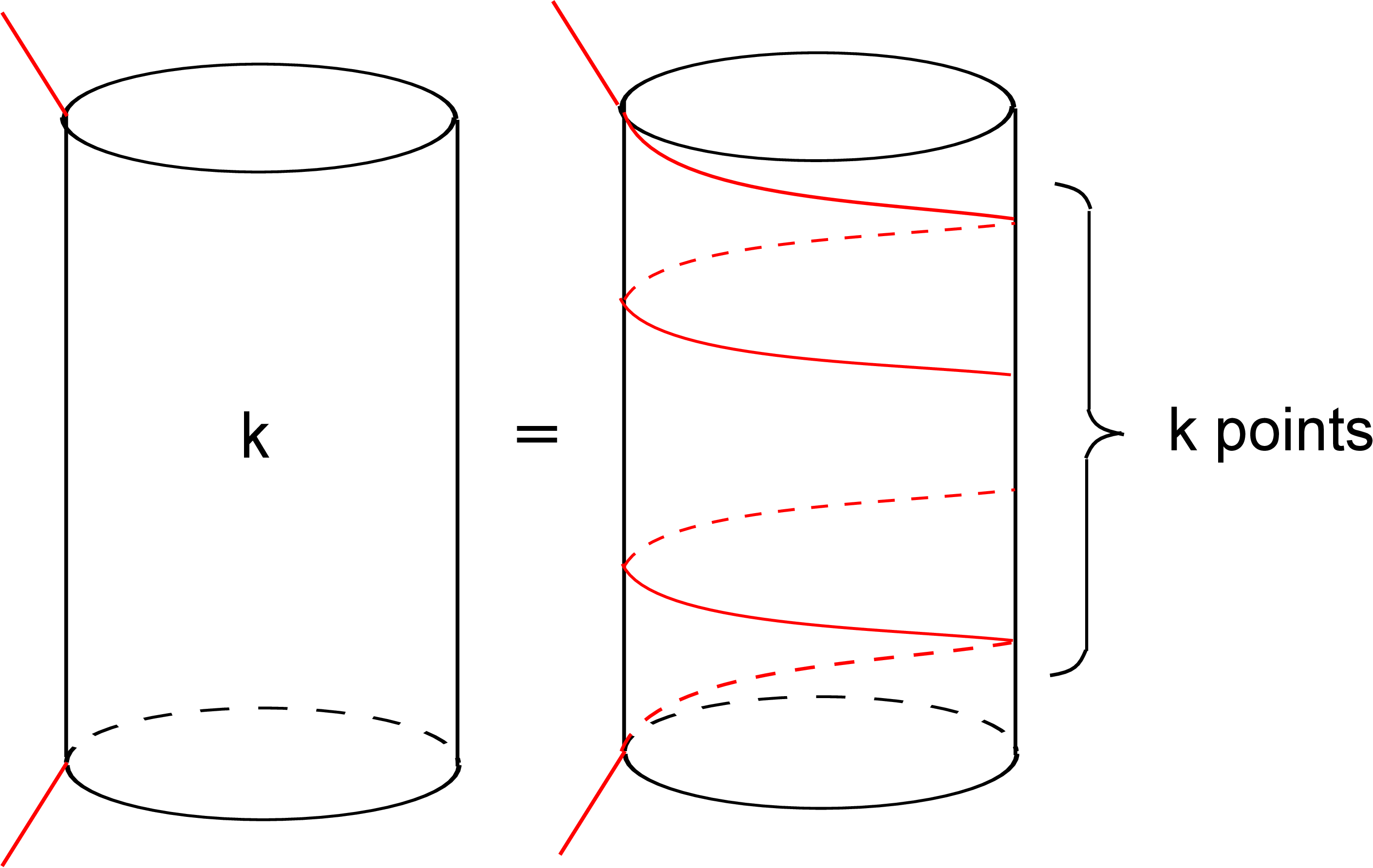}{3in} $$
\caption{}
\label{k}
\end{figure}

Note that $\tilde{P}$ and $\tilde{Q}$ respectively correspond to taking the $SL_2$-trace of the relations $w^nawa^{-1}w^{-1}a^{-1}w^{1-n}=a^{-1}w^{-1}awa$ and  $w^nawa^{-1}w^{-1}=a^{-1}w^{-1}awaw^{n-1}a$ in the fundamental group $\pi_1(S^3 \setminus K)$. Hence $\ve(\tP\star 1)= P, \ve(\tQ\star 1)= Q$, where \begin{eqnarray*}
P &=& x - x y + (-3 + x^2 + y^2) z - x y z^2 + z^3,\\
Q &=& S_{n-2}(y)+S_{n-3}(y)-S_{n-4}(y)-S_{n-5}(y)-S_{n-2}(y) \, x^2 \\
   && + \, \big( S_{n-1}(y)+S_{n-3}(y)+S_{n-4}(y) \big) \, xz-\big( S_{n-2}(y)+S_{n-3}(y) \big)\,z^2. \nonumber
\end{eqnarray*}
are the defining equations for the character variety of the $(-2, 3, 2n+1)$-pretzel knot $K$ as in Theorem \ref{universal}.

\subsection{Degrees}

\no{Let $\tCS$ be the skein of the twice punctured torus. Then $\tCS$ acts on $\CR[x,y,z]$, which is identified with the skein module of the genus 2 handle body.
Here $\CR= \BC[t^{\pm1 }]$.}

For a monomial $\fm:=y^k z^l$ define $\deg_y(\fm)= k$, $\deg_z(\fm)= l$ and $\deg_{yz}(\fm) = k+l$.
We linearly order the monomials $y^k z^l$ by the lexicographic order of the pair $(\deg_{yz},\deg_z)$.
Using this linear order, for a non-zero element in $\bD[y,z]$ (or in $\BC(M)[y,z]$) one can define its leading term, leading coefficient, and leading monomial.

In the discussion below, when talking about polynomials, we assume that the ground ring is either $\bD$ or $\ve(\bD)=\BC(M)$.

We say that two polynomials $f$ and $g$ has {\em equivalent leading term}, and write
$$ f \eqlt g,$$
if the leading term of $f$ is a unit times the leading term of $g$. Here unit means an invertible of the ground ring. For the case when the ground ring is $\BC(M)$, a unit is
a non-zero element.

\begin{lemma}\label{lem.11}
(a) There are polynomials $c_P, c_Q \in \BC(M)[y,z]$ with $\deg_{yz}(c_P) \le 3n-3$ and $\deg_{yz}(c_Q)\le 2n$ such that
$$
c_P P + c_Q Q  \eqlt y^{3n-2}.
$$
(b) There are polynomials $d_P, d_Q \in \BC(M)[y,z]$ with $\deg_{yz}(d_P) \le 2n-4$ and $\deg_{yz}(d_Q)\le n-1$ such that
$$
d_P P + d_Q Q  \eqlt  y^{2n-2}z.
$$
\end{lemma}
\begin{proof}

(a) In general, if $P= z^3 + az^2 + bz + c$ and $Q= dz^2 + ez + f$, then the resultant $\mathfrak r$ of $P$ and $Q$ with respect to $z$ is
$$ \mathfrak r= c_Q \, Q + c_P \, P,$$
where
$$ c_P= (-d^3\, b-d\, e^2+d^2\, a\, e+d^2\, f)\, z+2\, e\, d\, f-d^2\, b\, e+d^3\, c+d\, a\, e^2-a\, d^2\, f-e^3$$
\begin{align*}
c_Q = & (e^2-d\, f+d^2\, b-d\, a\, e)\, z^2+(a\, d^2\, b-d\, a^2\, e-f\, e-d^2\, c+a\, e^2)\, z\\
& +d\, a^2\, f+f^2+d^2\, b^2-e\, a\, f+d\, e\, c+e^2\, b-2\, d\, f\, b-d^2\, a\, c-d\, a\, e\, b.
\end{align*}
In our case, with explicit $P$ and $Q$, one can  check that
$c_P$ is a polynomial in $y,z$ with $\deg_{yz}(c_P) = 3n-3$, $c_Q$ is a polynomial in $y,z$ with $\deg_{yz}(c_Q)= 2n$. By Proposition \ref{fg}, $\mathfrak r$ is a polynomial in  $y$ with (equivalent) leading term $y^{3n-2}$.

(b) Let
$$
T= P \coeff(Q,z^2) - zQ = (ad-e)z^2+(bd-f)z+cd.
$$
Then $\deg_z(T)=2$. Let
$$ \mathfrak r'= T \coeff(Q,z^2) - Q \coeff(T, z^2)=(d(bd-f)-e(ad-e))z+cd^2-f(ad-e).$$
As in the proof of Proposition \ref{fg}, one can show that $\mathfrak r'$ is a polynomial in $y,z$ with (equivalent) leading term $y^{2n-2}z$.


Note that $\mathfrak r'=d^2P-(dz+ad-e)Q$. Let $d_P=d^2$ and $d_Q=-(dz+ad-e)$. Then $\mathfrak r'=d_P P+d_QQ$. With explicit $P$ and $Q$, it is easy to see that $\deg_{yz}(d_P) = 2n-4$ and $\deg_{yz}(d_Q)= n-1$. This completes the proof of the lemma.
\end{proof}

\def\fM{\frak M}
\def\fW{\frak U}
\def\fw{\frak u}

\def\al{\alpha}

\subsection{Filtration on $\bS$}  Recall that $\CS= \CR[x][y,z]/N$, where $N$ is the submodule defined using sliding relations, see Proposition \ref{N}.
It follows that $\bS= \CS \otimes_{\CR[x]}\bD$ is given by
 $$\bS= \bD[y,z]/\bN,$$ where $\bN= N \otimes_{\CR[x]}\bD$.

Let
$\scP$ be the $\bD$-span of $\tP\star y^k z^l, k, l \ge 0$, and
$\scQ$ be the $\bD$-span of $\tQ\star y^k z^l, k,l \ge 0$. By Proposition \ref{N},
\be
\scP, \scQ \subset \bN.
\label{eq.12}
\ee

Let $\fM_1=\{z\}$. For $k\ge 2$ let $\fM_k$ be the set $\{ y^l z^{k-l}, l = 0, \dots, k-1\} \cup \{ y^{k-2}\}$.

For $m \ge 1$ let
$$ \fM_{\le m} := \cup_{k=1}^m \fM_k,$$
and $\fW_m\subset \bD[y,z]$ be the $\bD$-span of $\fM_{\le m}$, which is a free finitely generated $\bD$-module. Then $\{ \fW_m, m \ge 1\}$ forms a filtration of $\bD[y,z]$:
$$ \fW_1 \subset \fW_2 \subset  \fW_3 \dots  , \quad \bigcup \fW_m = \bD[y,z].$$

This filtration induces a filtration on the quotient $\bS= \bD[y,z]/\bN$ as follows. Let $\scX_m:= \scX \cap \fW_m$ for $\scX = \bN, \scP, \scQ$, and
$$
\bS_m := \fW_m/\bN_m.
$$
Then $\bS_m$ is the set of elements of $\bS$ which can be represented by an element in $\fW_m$.

There is the natural embedding
$$
j_m: \bS_m \hookrightarrow \bS_{m+1}.
$$
and
$$
\bS = \bigcup_{m=1}^\infty \bS_m.
$$

We will show that if $m \ge 3n$, then $j_m$ is an isomorphism. This implies that $\bS$ is a finitely generated $\bD$-module.

Recall that for a $\bD$-homomorphism $f: V \to V'$,
$\ve(f)$ is the map $(V \overset{f}\to V') \otimes_{\bD} \ve(\bD)$.

\begin{lemma} \label{lem.100}
If $\ve(j_m)$ is surjective, then $j_m$ is surjective.
\end{lemma}
\begin{proof}
Since $\bS_m$ is a finitely generated module over the local ring $\bD$, surjectivity of $j_m$ follows from Nakayama's lemma as in the proof of Lemma \ref{lem.surj}.
\end{proof}

\subsection{On $\ve(\bS_m)$} Recall that $\bS_m = \fW_m/\bN_m$. Let $p: \bD[y,z]\to \BC(M)[y,z]$ be the algebra map which sends $t$ to $-1$.

Let $\fw_m := p(\fW_m)$. Then $\fw_m$ is the $\BC(M)$-vector space spanned by $\fM_{\le m}$. 
As $\bN_m$ is a subset of $\fW_m$, one has $p(\bN_m) \subset \fw_m$.

\begin{lemma} \label{lem.13}
One has a natural isomorphism
\be
\ve(\bS_m) = \fw_m/p(\bN_m).
\ee
\end{lemma}
\begin{proof}

Taking the tensor product of the exact sequence
$ \bN_m \to \fW_m \to \bS_m \to 0$ with $\BC(M)$ over $\bD$,
 one gets
 the following
commutative diagram with exact rows
\be \begin{CD}
 \bN_m @>\iota>>  \fW_m @>>>  \bS_m @>>>  0 \\
@VVq  V @VVp V @VVV \\
\ve(\bN_m) @> \ve(\iota)>>  \ve(\fW_m) @>>>  \ve(\bS_m) @>>>  0
\end{CD}
\label{cd.1}
\ee
where $q: \bN_m \to \ve(\bN_m)$ is the natural map. Since $q$ is surjective, one has
$$ \im(\ve(\iota))=  \im ( \ve(\iota) \circ q  )= \im (p \circ \iota)= p(\bN_m).$$
The exactness of the second row of Diagram \eqref{cd.1} shows that
$$ \ve(\bS_m) = \ve(\fW_m)/\im(\ve(\iota))= \ve(\fW_m)/p(\bN_m)= \fw_m/p(\bN_m).$$
This completes the proof of the lemma.
\end{proof}

\subsection{On the filtrations $\scP_m$ and  $\scQ_m$}

\begin{lemma} \label{lem.21}
 Suppose $f(y,z) \in \BC(M)[y,z]$ such that $\deg_{yz}f \le m$. Then
$$ \tP\star f  \in \scP_{m+3} \subset \bN_{m+3}, $$
$$ \tQ\star f \in \scQ_{m+n}\subset  \bN_{m+n}.$$
\end{lemma}

\begin{proof}
Recall that $H_2$ is  the cylinder $D^2 \times [0,1]$ minus two vertical tubes. Let $D_{**}$ be the top surface of $H_2$, which is the disk $D^2 \times{1}$ minus 2 holes. Let $a_0$ and $a_2$ be the straight arcs on $D_{**}$ depicted in Figure \ref{a2a0}.

To prove Lemma \ref{lem.21} we apply the upper bound and parity of $\deg_{yz}$ using the intersection number of link diagrams (on $D_{**}$) of $\tilde{P}\star f$ and $\tilde{Q}\star f$ with the arc $a_2$, and an upper bound for $\deg_y$ using
the intersection with $a_0$ as in \cite[Lemma 5.1]{Le13}.

\begin{figure}[htpb]
$$\psdraw{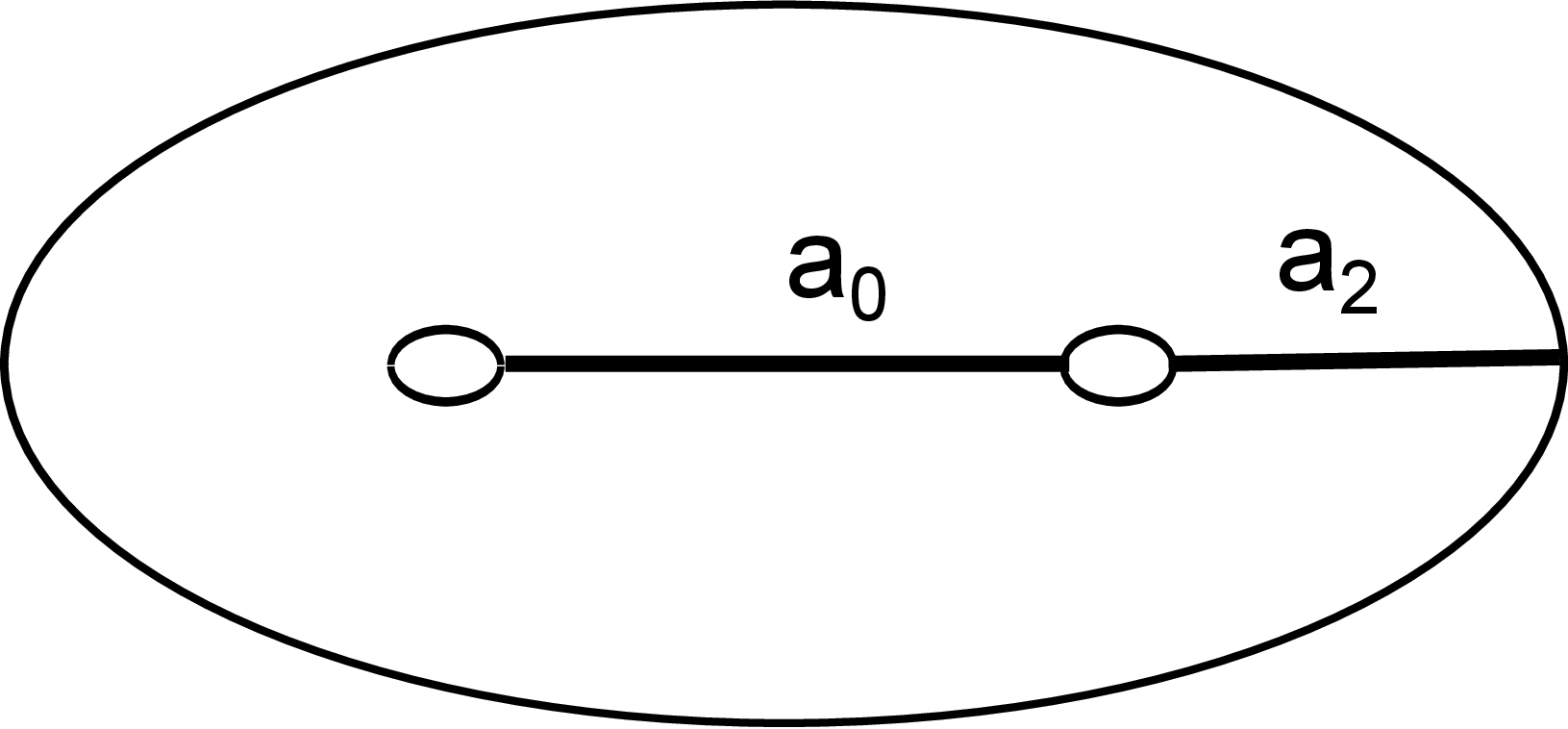}{2in} $$
\caption{The straight segments $a_0$ and $a_2$ on $D_{**}$}
\label{a2a0}
\end{figure}

We can assume that $f$ is a monomial, i.e. $f= y^k z^l$ for some $k,l \ge 0$. We will only show that $\tP\star f  \in \scP_{k+l+3}$. (The proof that $\tQ\star f \in \scQ_{k+l+n}$ is similar.)

We have $\tP \star y^k z^l = \tP_1 \star y^k z^l - \tP_2 \star y^k z^l$, where $\tP_1$ and $\tP_2$ are depicted in Figure \ref{P1P2}. Let $k_0$ (resp. $k_2$) be the intersection number of the link diagram (on $D_{**}$) of $\tP_1 \star y^k z^l$ with $a_0$ (resp. $a_2$). It is easy to see that $k_0=k+2$ and $k_2=k+l+3$.

Suppose $c_{a,b} \, y^az^b$ is a monomial of $\tP_1 \star y^k z^l \in \BC[x][y,z]$, where $a,b \ge 0$ and $c_{a,b} \in \BC[x]$. By \cite[Lemma 5.1]{Le13}, one has
$$(\deg c_{a,b}) + a \le k_0=k+2,$$
$$a+b \le k_2=k+l+3, \quad \text{and} \quad a+b \equiv k+l+3 \pmod{2}.$$

If $b>0$, then since $a+b \le k+l+3$ one has $c_{a,b} \, y^az^b \in \fW_{k+l+3}$.

If $b=0$, then $a \le k+2<k+l+3$. Since $a \equiv k+l+3 \pmod{2}$, we must have $a \le k+l+1$ which implies that $c_{a,b} \, y^a \in \fW_{k+l+3}$.

Hence $\tP_1 \star y^k z^l \in \fW_{k+l+3}$. Similarly, $\tP_2 \star y^k z^l \in \fW_{k+l+3}$. It follows that $\tP \star y^k z^l = \tP_1 \star y^k z^l - \tP_2 \star y^k z^l$ is in $\fW_{k+l+3} \cap \scP= \scP_{k+l+3}$.
\end{proof}

\subsection{Surjectivity of $\ve(j_m)$}
\begin{lemma}If $m \ge 3n$, then $\ve(j_{m-1})$ is surjective.
\label{lem.87}
\end{lemma}
\begin{proof}
By Lemma \ref{lem.13}, $\ve(\bS_m) \cong \fw_m/p(\bN_m)$, and
$$\ve(j_{m-1}) :  \fw_{m-1}/p(\bN_{m-1}) \to \fw_m/p(\bN_m)$$
descends from the embedding $\fw_{m-1} \hookrightarrow \fw_{m}$. It follows that $\ve(j_{m-1})$ is surjective if and only
$$ \fw_m \subset \fw_{m-1} + p(\bN_m),$$
which, because of \eqref{eq.12}, will follow from
\be
\fw_m \subset \fw_{m-1} +  p(\scP_m) + p(\scQ_m).
\label{eq.54}
\ee
\no{

Note that \eqref{eq.54} is equivalent to
\be
\fw_m \subset  p(\scQ_m)    \mod {  \left( \fw_{m-1} + p(\scP_m)  \right)}
\label{eq.55}
\ee
}
Hence to prove Lemma \ref{lem.87}, one only needs to prove \eqref{eq.54}.

\smallskip

\underline{Claim 1}: One has
$$
 \fw_m \ \subset \ \left [\fw_{m-1} + p(\scP_m)  + (\text{$\BC(M)$-span of  $\{ y^{m-2} z^2, y^{m-1}z, y^{m-2}\}$}) \right],
$$
i.e.
modulo $(\fw_{m-1} + p(\scP_m)) $, $\fw_m$ is $\BC(M)$-spanned by $y^{m-2} z^2, y^{m-1}z, y^{m-2}$.

\begin{proof} (of Claim 1) Modulo $\fw_{m-1}$, $\fw_{m}$ is $\BC(M)$-spanned by
$$\fM_m= \{z^m , y z^{m-1}, \dots, y^{m-2}z^2, y^{m-1} z\} \cup \{ y^{m-2}\}.$$
Note that $P \in p(\scP_3)$, with leading term $z^3$.
Modulo $p(\scP_m) $, any monomial in $\fM_m$  with $z$-degree $\ge 3$ can be reduced to a sum of terms with $z$-degree less than 3. Since the elements of $\fM_m$ with $z$-degree less than 3
are $y^{m-2} z^2, y^{m-1}z$ and $y^{m-2}$, Claim 1 follows.
\end{proof}

\underline{Claim 2}: One has
$$
 y^{m-2} z^2 \ \in \ \left [\fw_{m-1}  + p(\scQ_m)  + (\text{$\BC(M)$-span of  $\{ y^{m-1}z, y^{m-2}\}$}) \right].
 \label{eq.235}
$$

\begin{proof} (of Claim 2) From Lemma \ref{lem.21} one has
$$Q \, y^{m-n}=p(\tilde{Q} \star y^{m-n})\in p(\scQ_m).$$
Note that $Q \, y^{m-n}$ has leading monomial $y^{m-2} z^2$. Hence modulo $p(\scQ_m)$, $y^{m-2} z^2$ can be reduced to a sum of terms in $\fw_m$ such that each of them is $<y^{m-2} z^2$ (by the lexicographic order of the pair $(\deg_{yz},\deg_z)$). Since a term in $\fw_m$ that is $<y^{m-2} z^2$ is either $y^{m-1}z, y^{m-2}$, or a term in $\fw_{m-1}$, Claim 2 follows.
\end{proof}

\underline{Claim  3}: One has
$$
 y^{m-1}z \ \in \ \left [\fw_{m-1} + p(\scP_m) + p(\scQ_m)  + (\text{$\BC(M)$-span of  $\{ y^{m-2}\}$}) \right].
 \label{eq.235a}
$$

\begin{proof} (of Claim 3)
From Lemma \ref{lem.21} one has $$(d_PP + d_QQ)y^{m-(2n-1)}=\ve(\tilde{P} \star d_P y^{m-(2n-1)})+\ve(\tilde{Q} \star d_Q y^{m-(2n-1)}) \in p(\scP_m) + p(\scQ_m).$$
By Lemma \ref{lem.11}(b), one has
$$ (d_PP + d_QQ)y^{m-(2n-1)} \eqlt y^{m-1} z.$$
Hence modulo $p(\scP_m) + p(\scQ_m) $, $y^{m-1} z$ can be reduced to a sum of terms in $\fw_m$ such that each of them is $<y^{m-1}z$. Since a term in $\fw_m$ that is $<y^{m-1} z$ is either $y^{m-2}$ or a term in $\fw_{m-1}$, Claim 3 follows.
\end{proof}

\underline{Claim  4}: One has
$$
 y^{m-2} \ \in \ \left [\fw_{m-1} + p(\scP_m) + p(\scQ_m)  \right].
$$

\begin{proof} (of Claim 4) From Lemma \ref{lem.21} one has
$$(c_P P + c_QQ)y^{m-3n}=\ve(\tilde{P} \star c_P y^{m-3n})+\ve(\tilde{Q} \star c_Q y^{m-3n}) \in p(\scP_m) + p(\scQ_m).$$ By Lemma \ref{lem.11}(a), one has
$$ (c_PP + c_QQ)y^{m-3n} \eqlt y^{m-2}.$$
Modulo $p(\scP_m) + p(\scQ_m) $, $y^{m-2}$ can be reduced to a sum of terms in $\fw_m$ such that each of them is $<y^{m-2}$. Since a term in $\fw_m$ that is $<y^{m-2}$ is a term in $\fw_{m-1}$, Claim 4 follows.
\end{proof}

From Claims 1, 2, 3 and 4, one has \eqref{eq.54}. Lemma \ref{lem.87} then follows.
\end{proof}

\subsection{Proof of Proposition \ref{localized}}

By   Lemma \ref{lem.100} and Lemma \ref{lem.87}, the map $j_m$ is an isomorphism for $m \ge 3n$. It follows that
$$\bS_{3n} = \bS_{3n+1}= \bS_{3n+2}=\dots.$$
Hence $\bS= \bS_{3n}$, which is finitely generated over $\bD$.

\appendix

\numberwithin{equation}{section}
\numberwithin{figure}{section}

\section{Character varieties of two-bridge knots}

\label{pf}

We first review the description of character varieties of two-bridge knots from \cite{Le93}.
Suppose $K=\fb(p,m)$ is a two-bridge knot. Let $X=S^3 \setminus K$. Then $$\pi_1(X)=\langle a,b \mid wa=bw
\rangle,$$ where both $a$ and $b$ are meridians. The word $w$ has the form $a^{\varepsilon_1}b^{\varepsilon_2} \cdots a^{\varepsilon_{p-2}}b^{\varepsilon_{p-1}}$, where $\varepsilon_j:=(-1)^{\lfloor jm/p\rfloor}$. In particular, if
we read $w$ from right to left and interchange $a$ and $b$ then we get $w$ again. For example, $\mathfrak{b}(p,1)$ is the torus knot $T(2,p)$, and in this case $w=(ab)^{d}$, where $d:=(p-1)/2$.

We adopt the convention that if $r: \pi_1(X) \to SL_2(\BC)$ is a representation and $u$ is a word then we write $u$ also for $r(u)$ and $|u|$ for $\tr r(u)$. If $u$ is a word then $u'$ denotes the word obtained from $u$ by deleting the two letters at the two ends.

Let  $x:=|a|=|b|$ and $z:=|ab|$. It was shown in \cite{Le93} that the non-abelian character variety, i.e. the set of characters of non-abelian representations, of $\pi_1(X)$ is the zero set of the polynomial $${\Phi}_{(p,m)}(x,z)=|w|-|w'|+\dotsb+(-1)^{d-1}|w^{(d-1)}|+(-1)^d.$$ Moreover ${\Phi}_{(p,m)}(x,z)$ is a polynomial in $\BZ[x^2,z]$ with $z$-leading term $z^{d}$.

\subsection{Irreducibility over $\BQ$}

Let ${\Phi}_d(x,z)={\Phi}_{(p,1)}(x,z)$, where $d=(p-1)/2$.
It was shown in \cite[Proposition 4.3.1]{Le93} (also see below) that ${\Phi}_d(x,z)$ does not depend on $x$.

\begin{proposition}\label{phi_n}
The polynomial ${\Phi}_d(z)$ is irreducible over $\BQ$ if and only if $p=2d+1$ is prime.
\end{proposition}

\begin{proof}
It is immediate from \cite[Proposition 4.3.1]{Le93} that ${\Phi}_d(z)=S_d(z)-S_{d-1}(z)$, where $S_n(z)$'s are the Chebyshev polynomials defined by $S_0(z)=1,\,S_1(z)=z$ and $S_{n+1}(z)=z\,S_n(z)-S_{n-1}(z)$. By similar arguments as in the proof of Lemma \ref{T'} one can show that ${\Phi}_d(z)$ is an integer polynomial of degree $d$ with exactly $d$ roots given by $z= 2\cos\big(\frac{2j+1}{2d+1}\pi\big)$, $0\le j\le d-1$. It follows that the splitting field of ${\Phi}_d(z)$ is $\BQ(\cos \eta)$, where $\eta:=\pi/p$. Hence $\Phi_d(z)$ is irreducible over $\BQ$ if and only if the extension field degree  $[\BQ(\cos\eta):\BQ]$ is exactly the degree of ${\Phi}_d$, which is $d$.

Note that $\cos\eta=(e^{i\eta}+e^{-i\eta})/2$. We need to study the extension field $\BQ(e^{i\eta})/\BQ$. It is well-known that the minimal polynomial over $\BQ$ of $e^{i\eta}$ is the cyclotomic polynomial
$$C_{2p}(t)=\prod_{1\le j\le 2p,\ \gcd (j,2p)=1}(t-e^{j\pi i/p}),$$
see e.g.  \cite{La}. This is an integer polynomial whose  degree is  $\varphi(2p)=\varphi(p)$, where $\varphi$ is the Euler totient function. Thus the degree of the extension field is $[\BQ(e^{i\eta}):\BQ]=\varphi(p)$. From the identity $(t-e^{i\eta })(t-e^{-i\eta})=t^2-2(\cos\eta) t+1$, we see that $[\BQ(e^{i\eta}):\BQ(\cos\eta)]=2$, thus $[\BQ(\cos\eta):\BQ]=\varphi(p)/2$.
Therefore ${\Phi}_d(z)$ is irreducible over $\BQ$ if and only if $\varphi(p)=p-1$, which occurs if and only if $p$ is prime.
\end{proof}

\begin{proposition} One has ${\Phi}_{(p,m)}(0,z)={\Phi}_{(p,1)}(z)$. Hence if ${\Phi}_{(p,1)}(z)$ is irreducible in
$\BQ[z]$ then ${\Phi}_{(p,m)}(x,z)$ is irreducible in $\BQ[x,z]$.
\label{x=0}
\end{proposition}

\begin{proof}
If $x=|a|=|b|=0$ then $a^{-1}=-a$ and $b^{-1}=-b$. (This follows from the Cayley-Hamilton theorem applying for matrices in $SL_2(\BC)$: $a+a^{-1}=|a|\,I_{2 \times 2}$, where $I_{2 \times 2}$ is the $2 \times 2$ identity matrix.)

Recall that ${\Phi}_{(p,m)}(x,z)=|w|-|w'|+\dotsb+(-1)^{d-1}|w^{(d-1)}|+(-1)^d$. From the definition of the word $w$, it is easy to see that $a^{-1}$ and $b^{-1}$ appear in pairs in $w$. This is also true for $a^{-1}$ and $b^{-1}$ in each word $w^{(j)}$, $0 \le j \le d-1$, hence $w^{(j)}$ does not change if one simultaneously replaces $a^{-1}$ by $a$ and $b^{-1}$ by $b$. Thus $w^{(j)}=(ab)^{d-j}$. Note that for the torus knot $\fb(2d+1,1)$ we have $w=(ab)^d$, hence the proposition follows.
\end{proof}

\def\Gal{{\mathrm {Gal}}}

\def\id{\mathrm {id}}
\def\disc{\mathrm {disc}}

\subsection{Irreducibility over $\BC$} For a word $u$, let $\overleftarrow{u}$ be the word obtained from $u$ by writing the letters in $u$ in reversed order. Then, by \cite[Lemma 3.2.2]{Le93}, $|\overleftarrow{u}|=|u|$ for any word $u$ in 2 letters $a$ and $b$.

Suppose $\nu_1, \nu_2, \dots, \nu_d \in \{-1,1\}$. Let $\nu_{d+j}:=\nu_{d+1-j}$ for $j=1, \dots, d$. Let $$w_j=a^{\nu_j}b^{\nu_{j+1}} \dots a^{\nu_{2d-j}}b^{\nu_{2d+1-j}}.$$ Then $w_1=a^{\nu_1}b^{\nu_2} \dots a^{\nu_{2d-1}}b^{\nu_{2d}}$ and $w_{j+1}=(w_j)'$.

Let $\mu_j:=\nu_j\nu_{j+1}$ for $j=1, \dots, d$. Note that $\mu_d=1$. Let $c_j$ be the number of $\mu_k=-1$ among $\mu_j, \dots, \mu_d$.

Recall that $x=|a|=|b|$ and $z=|ab|$. Let $X:=x^2.$

\begin{proposition}
$|w_j|$ is a polynomial in $X,z$ of total degree $d+1-j$ and
$$
|w_j|=z^{d+1-j-c_j}(z-X)^{c_j}+l.o.t.
$$
Here l.o.t. is the term of total degree $<d+1-j$.
\label{lot}
\end{proposition}

\begin{proof}
Let $u_j:=w_{j+1}a^{\nu_j}=a^{\nu_{j+1}}  \dots b^{\nu_{j+1}}  a^{\nu_j}$ and $v_j:=b^{\nu_j}w_{j+1}=b^{\nu_j}a^{\nu_{j+1}}   \dots   b^{\nu_{j+1}}$ for $j=1, \dots, d$, where $w_{d+1}:=1$. We will show that

1) $x|u_j|$ and $x|v_j|$ are polynomials in $X,z$ of total degree $\le d+1-j$,

2) $|w_j|$ is a polynomial in $X,z$ of total degree $d_j$ and $$|w_j|=z^{d+1-j-c_j}(z-X)^{c_j}+l.o.t.,$$
by induction on $1 \le j \le d$, beginning with $j=d$ which is obvious.

Suppose $j \le d-1$. Consider the following 2 cases: $\nu_j\nu_{j+1}=1$ and $\nu_j\nu_{j+1}=-1$.

\textit{Case 1: $\nu_j\nu_{j+1}=1$,} i.e. $\nu_{j}=\nu_{j+1}$. Then $c_j=c_{j+1}$ and
\begin{eqnarray*}
x|u_j| &=& x|(a^{\nu_{j+1}}b^{\nu_{j+2}} \dots a^{\nu_{j+2}}b^{\nu_{j+1}})a^{\nu_{j+1}}| \\
&=& x|(a^{\nu_{j+1}}b^{\nu_{j+2}} \dots a^{\nu_{j+2}}b^{\nu_{j+1}})(xI_{2 \times 2}-a^{-\nu_{j+1}})| \\
&=& x^2|a^{\nu_{j+1}}b^{\nu_{j+2}} \dots a^{\nu_{j+2}}b^{\nu_{j+1}}|-x|b^{\nu_{j+2}} \dots a^{\nu_{j+2}}b^{\nu_{j+1}}| \\
&=& x^2|w_{j+1}|-x|\overleftarrow{v_{j+1}}|= x^2|w_{j+1}|-x|v_{j+1}|, \\
x|v_j| &=& x|b^{\nu_{j+1}}(a^{\nu_{j+1}}b^{\nu_{j+2}} \dots a^{\nu_{j+2}}b^{\nu_{j+1}})| \\
&=& x|(xI_{2 \times 2}-b^{-\nu_{j+1}})(a^{\nu_{j+1}}b^{\nu_{j+2}} \dots a^{\nu_{j+2}}b^{\nu_{j+1}})| \\
&=&x^2|a^{\nu_{j+1}}b^{\nu_{j+2}} \dots a^{\nu_{j+2}}b^{\nu_{j+1}}|-x|a^{\nu_{j+1}}b^{\nu_{j+2}} \dots a^{\nu_{j+2}}| \\
&=& x^2|w_{j+1}|-x|\overleftarrow{u_{j+1}}|=x^2|w_{j+1}|-x|u_{j+1}|, \\
|w_j| &=& |(a^{\nu_{j+1}}b^{\nu_{j+1}})a^{\nu_{j+2}} \dots b^{\nu_{j+2}}(a^{\nu_{j+1}}b^{\nu_{j+1}})|\\
&=& |(a^{\nu_{j+1}}b^{\nu_{j+1}})a^{\nu_{j+2}} \dots b^{\nu_{j+2}}(zI_{2 \times 2}-(a^{\nu_{j+1}}b^{\nu_{j+1}})^{-1})|\\
&=& z |a^{\nu_{j+1}}b^{\nu_{j+1}}a^{\nu_{j+2}} \dots b^{\nu_{j+2}}|-|a^{\nu_{j+2}} \dots b^{\nu_{j+2}}|\\
&=& z |b^{\nu_{j+1}}a^{\nu_{j+2}} \dots b^{\nu_{j+2}}a^{\nu_{j+1}}|-|w_{j+2}|\\
&=& z|\overleftarrow{w_{j+1}}|-|w_{j+2}|= z|w_{j+1}|-|w_{j+2}|.
\end{eqnarray*}
It follows that 1) and 2) hold true for $j$, by induction hypothesis.

\textit{Case 2: $\nu_j\nu_{j+1}=-1$,} i.e. $\nu_j=-\nu_{j+1}$. Then $c_j=c_{j+1}+1$ and
\begin{eqnarray}
x|u_j| &=& x|(a^{\nu_{j+1}}b^{\nu_{j+2}} \dots a^{\nu_{j+2}}b^{\nu_{j+1}})a^{-\nu_{j+1}}|=x|b^{\nu_{j+2}} \dots a^{\nu_{j+2}}b^{\nu_{j+1}}|\nonumber\\
&=& x|\overleftarrow{v_{j+1}}| =x|v_{j+1}|,\nonumber\\
x|v_j| &=& x|b^{-\nu_{j+1}}(a^{\nu_{j+1}}b^{\nu_{j+2}} \dots a^{\nu_{j+2}}b^{\nu_{j+1}})|=x|a^{\nu_{j+1}}b^{\nu_{j+2}} \dots a^{\nu_{j+2}}|\nonumber\\
&=& x|\overleftarrow{u_{j+1}}|=x|u_{j+1}|,\nonumber\\
|w_j| &=& |a^{-\nu_{j+1}}b^{\nu_{j+1}}a^{\nu_{j+2}} \dots b^{\nu_{j+2}}a^{\nu_{j+1}}b^{-\nu_{j+1}}| \nonumber\\
&=&|(x I_{2 \times 2}-a^{\nu_{j+1}})b^{\nu_{j+1}}a^{\nu_{j+2}} \dots b^{\nu_{j+2}}a^{\nu_{j+1}}(xI_{2 \times 2}-b^{\nu_{j+1}})| \nonumber\\
&=& x^2|b^{\nu_{j+1}}a^{\nu_{j+2}} \dots b^{\nu_{j+2}}a^{\nu_{j+1}}|+|a^{\nu_{j+1}}b^{\nu_{j+1}}a^{\nu_{j+2}} \dots b^{\nu_{j+2}}a^{\nu_{j+1}}b^{\nu_{j+1}}| \nonumber\\
&& - \, x|a^{\nu_{j+1}}b^{\nu_{j+1}}a^{\nu_{j+2}} \dots b^{\nu_{j+2}}a^{\nu_{j+1}}|-x|b^{\nu_{j+1}}a^{\nu_{j+2}} \dots b^{\nu_{j+2}}a^{\nu_{j+1}}b^{\nu_{j+1}}|.
\label{tt}
\end{eqnarray}

We have
\begin{equation}
|b^{\nu_{j+1}}a^{\nu_{j+2}} \dots b^{\nu_{j+2}}a^{\nu_{j+1}}|=|\overleftarrow{w_{j+1}}|=|w_{j+1}|.
\label{t0}
\end{equation}
By Case 1,
\begin{equation}
|a^{\nu_{j+1}}b^{\nu_{j+1}}a^{\nu_{j+2}} \dots b^{\nu_{j+2}}a^{\nu_{j+1}}b^{\nu_{j+1}}|=z|w_{j+1}|-|w_{j+2}|.
\label{t11}
\end{equation}
 We have
\begin{eqnarray}
x|a^{\nu_{j+1}}b^{\nu_{j+1}}a^{\nu_{j+2}} \dots b^{\nu_{j+2}}a^{\nu_{j+1}}| &=&x |(xI_{2 \times 2}-a^{-\nu_{j+1}}) b^{\nu_{j+1}}a^{\nu_{j+2}} \dots b^{\nu_{j+2}}a^{\nu_{j+1}}| \nonumber\\
&=&x^2|b^{\nu_{j+1}}a^{\nu_{j+2}} \dots b^{\nu_{j+2}}a^{\nu_{j+1}}|-x|b^{\nu_{j+1}}a^{\nu_{j+2}} \dots b^{\nu_{j+2}}| \nonumber\\
&=&x^2|\overleftarrow{w_{j+1}}|-x|v_{j+1}|=x^2|w_{j+1}|-x|v_{j+1}|.
\label{t21}
\end{eqnarray}
Similarly,
\begin{equation}
x|b^{\nu_{j+1}}a^{\nu_{j+2}} \dots b^{\nu_{j+2}}a^{\nu_{j+1}}b^{\nu_{j+1}}|=x^2|w_{j+1}|-x|u_{j+1}|.
\label{t3}
\end{equation}
From \eqref{tt}, \eqref{t0}, \eqref{t11}, \eqref{t21} and \eqref{t3}, we get
\begin{equation}
|w_j|=(z-x^2)|w_{j+1}|-|w_{j+2}|+x|u_{j+1}|+x|v_{j+1}|.
\label{t}
\end{equation}
Hence, by induction hypothesis, $|w_j|$ is a polynomial in $X,z$ of total degree $d+1-j$ and
$$|w_j|=(z-X)z^{d-j-c_{j+1}}(z-X)^{c_{j+1}}+l.o.t.=z^{d+1-j-c_j}(z-X)^{c_j}+l.o.t.$$
where l.o.t. is the term of total degree $<d+1-j$, since $c_j=c_{j+1}+1.$
\end{proof}

Applying Proposition \ref{lot} with $\nu_j=\ve_j=(-1)^{\lfloor jm/p\rfloor}$, $${\Phi}_{(p,m)}(x,z)=|w|-|w'|+\dots+(-1)^{d-1}|w^{(d-1)}|+(-1)^d$$
is a polynomial in $X,z$ of total degree $d=\frac{p-1}{2}$.

Let $\Gamma_{(p,m)}(X,z):={\Phi}_{(p,m)}(x,z) \in \BZ[X,z]$. Then, also by Proposition \ref{lot},
$$
\Gamma_{(p,m)}(X,z)=z^{d-c}(z-X)^c+l.o.t.,
$$
where l.o.t. is the term of total degree $<d$ and $c$ is the number of $\mu_k=-1$ among $\mu_1, \dots, \mu_d$. Note that $c=\frac{m-1}{2}$, see e.g. \cite{Bur}.

\begin{corollary}
 $(0,\frac{p-1}{2})$ and $(\frac{m-1}{2},\frac{p-m}{2})$ are vertices of the Newton polygon of the polynomial $\Gamma_{(p,m)}(X,z) \in \BZ[X,z]$.
 \label{Gamma}
\end{corollary}

\begin{theorem}
(i) Suppose $\Phi_{(p,m)}(x,z)$ is irreducible in $\BQ[x,z]$ and $\gcd(\frac{p-1}{2},\frac{m-1}{2})=1$. Then $\Phi_{(p,m)}(x,z)$ is irreducible in $\BC[x,z]$.

(ii) Suppose $p$ is prime and $\gcd(\frac{p-1}{2},\frac{m-1}{2})=1$. Then $\Phi_{(p,m)}(x,z)$ is irreducible in $\BC[x,z]$.
\label{prime}
\end{theorem}

\begin{proof}
(i) Suppose $\Phi_{(p,m)}(x,z)$ is irreducible in $\BQ[x,z]$. Then $\Gamma_{(p,m)}(X,z)$ is irreducible in $\BQ[X,z]$. Note that $(0,\frac{p-1}{2})$ and $(\frac{m-1}{2},\frac{p-m}{2})$ are vertices of the Newton polygon of the polynomial $\Gamma_{(p,m)}(X,z) \in \BZ[X,z]$ by Corollary \ref{Gamma}, and $\gcd \left( (0,\frac{p-1}{2}), (\frac{m-1}{2},\frac{p-m}{2}) \right)=1$ since $\gcd(\frac{p-1}{2},\frac{m-1}{2})=1$. Hence \cite[Proposition 3]{BCG} implies that $\Gamma_{(p,m)}(X,z)$ is irreducible in $\BC[X,z]$.

Assume that $\Phi_{(p,m)}(x,z)$ is reducible in $\BC[x,z]$ and $f(x,z)$ is an irreducible factor of $\Phi_{(p,m)}(x,z)$. Write $f(x,z)=g(X,z)+xh(X,z)$, where $g,h \in \BC[X,z]$. If $h \equiv 0$ then $g(X,z)=f(x,z)$ is an irreducible factor of $\Gamma_{(p,m)}(X,z)$ in $\BC[X,z]$ and the total degree of $g(X,z)$ is less than that of $\Gamma_{(p,m)}(X,z)$, a contradiction since $\Gamma_{(p,m)}(X,z)$ is irreducible in $\BC[X,z]$. Hence $h \not\equiv 0$. Note that $f(-x,y)=g(X,z)-xh(X,z)$ is also an irreducible factor of  $\Phi_{(p,m)}(x,z)$ and $f(-x,z) \not= f(x,z)$. It follows that $f(x,y)f(-x,y) \in \BC[X,z]$ is a factor of $\Gamma_{(p,m)}(X,z)$. Since $\Gamma_{(p,m)}(X,z)$ is irreducible in $\BC[X,z]$, we must have $f(x,z)f(-x,z)=\Gamma_{(p,m)}(X,z)$. In particular $f^2(0,z)=\Gamma_{(p,m)}(0,z)=\Phi_{(p,m)}(0,z)$. This is impossible since $\Phi_{(p,m)}(0,z)$ does not have repeated factors, according to the proof of Proposition \ref{phi_n}. Hence $\Phi_{(p,m)}(x,z)$ is irreducible in $\BC[x,z]$.

(ii) Since $p$ is prime, by Propositions \ref{phi_n} and \ref{x=0}, $\Phi_{(p,m)}(x,z)$ is irreducible in $\BQ[x,z]$. The conclusion follows from Part (i).
\end{proof}

\end{document}